\documentclass[preprint]{elsarticle}

\RequirePackage{geometry}
\global\bibsep=0pt

\usepackage{framed,multirow}

\usepackage{latexsym}

\usepackage{url}
\usepackage{xcolor}

\journal{Communications in Computational Physics}

\usepackage{graphicx}
\usepackage{lipsum}
\usepackage{amsfonts}
\usepackage{graphicx}
\usepackage{epstopdf}
\usepackage{algorithmic}
\ifpdf
\DeclareGraphicsExtensions{.eps,.pdf,.png,.jpg}
\else
\DeclareGraphicsExtensions{.eps}
\fi

\usepackage{multicol}        
\usepackage[bottom]{footmisc}

\usepackage{amssymb}
\usepackage{amsmath}
\usepackage{amsthm}
\usepackage{wasysym}
\usepackage{bbm}

\usepackage{bm}
\usepackage{xspace}

\usepackage{booktabs}
\usepackage{tabularx}

\usepackage{caption}
\usepackage{subfig}

\usepackage{color}

\usepackage{hyperref}
\usepackage[capitalize]{cleveref}

\usepackage{xargs}  

\usepackage[normalem]{ulem} 

\renewcommand{\vec}[1]{{\bf #1}}

\def\@thmcountersep{.}

\newtheorem{theorem}{Theorem}

\newtheorem{definition}{Definition}
\newtheorem{remark}{Remark}

\usepackage{tikz}
\usetikzlibrary{patterns,arrows,decorations.pathreplacing}
\tikzset{every label/.style={font=\footnotesize,inner sep=1pt}}
\newcommand{\cellcenter}[4][]{\node[label={below:#4},#1] at (#2) (#3) {}}
\newcommand{\interface}[4][]{\node[label={below:#4},#1] at (#2) (#3) {};}
\newcommand{\source}[1]{s_{#1}}
\newcommand{\sourcedef}[1]{\footnotesize $\rho_{#1}^{\rm rec}g_{#1}^{\rm int}$}
\makeatletter
\def\pt@get#1#2{
\tikz@scan@one@point\pgfutil@firstofone#2\relax%
\csname pgf@x#1\endcsname=\pgf@x%
\csname pgf@y#1\endcsname=\pgf@y%
}
\tikzset{
parabola through/.style={
to path={{[x={(\pgf@xc,\pgf@yc)}, y=\parabola@y, shift=(\tikztostart)]
-- (0,0) .. controls (1/3,1/3) and (2/3,1/3) .. (1,0) \tikztonodes}--(\tikztotarget)}
},
parabola through/.prefix code={
\pt@get{a}{(\tikztostart)}\pt@get{b}{#1}\pt@get{c}{(\tikztotarget)}%
\advance\pgf@xb by-\pgf@xa\advance\pgf@yb by-\pgf@ya%
\advance\pgf@xc by-\pgf@xa\advance\pgf@yc by-\pgf@ya%
\pgfmathsetmacro\parabola@y{(\pgf@yc-\pgf@xc/\pgf@xb*\pgf@yb)%
/(\pgf@xb-\pgf@xc)*\pgf@xc}%
}
}
\makeatother

\newcolumntype{L}[1]{>{\raggedright\arraybackslash}p{#1}} 
\newcolumntype{C}[1]{>{\centering\arraybackslash}p{#1}} 
\newcolumntype{R}[1]{>{\raggedleft\arraybackslash}p{#1}} 
\newcolumntype{Y}{>{\centering\arraybackslash}X} 
\newcolumntype{Z}{>{\raggedleft\arraybackslash}X} 

\def\@thmcountersep{.}

\makeatletter
\newcommand{\manuallabel}[2]{\def\@currentlabel{#2}\label{#1}}
\makeatother

\newcounter{subcount}

\newcommand{\refb}[1]{#1}

\DeclareMathAlphabet{\mathitbf}{OML}{cmm}{b}{it}

\makeatletter
\newcommand\RedeclareMathOperator{%
\@ifstar{\def\rmo@s{m}\rmo@redeclare}{\def\rmo@s{o}\rmo@redeclare}%
}
\newcommand\rmo@redeclare[2]{%
\begingroup \escapechar\m@ne\xdef\@gtempa{{\string#1}}\endgroup
\expandafter\@ifundefined\@gtempa
{\@latex@error{\noexpand#1undefined}\@ehc}%
\relax
\expandafter\rmo@declmathop\rmo@s{#1}{#2}}
\newcommand\rmo@declmathop[3]{%
\DeclareRobustCommand{#2}{\qopname\newmcodes@#1{#3}}%
}
\@onlypreamble\RedeclareMathOperator
\makeatother

\DeclareMathSymbol{\alpha}{\mathalpha}{letters}{"0B}
\DeclareMathSymbol{\beta}{\mathalpha}{letters}{"0C}
\DeclareMathSymbol{\gamma}{\mathalpha}{letters}{"0D}
\DeclareMathSymbol{\delta}{\mathalpha}{letters}{"0E}
\DeclareMathSymbol{\epsilon}{\mathalpha}{letters}{"0F}
\DeclareMathSymbol{\zeta}{\mathalpha}{letters}{"10}
\DeclareMathSymbol{\eta}{\mathalpha}{letters}{"11}
\DeclareMathSymbol{\theta}{\mathalpha}{letters}{"12}
\DeclareMathSymbol{\iota}{\mathalpha}{letters}{"13}
\DeclareMathSymbol{\kappa}{\mathalpha}{letters}{"14}
\DeclareMathSymbol{\lambda}{\mathalpha}{letters}{"15}
\DeclareMathSymbol{\mu}{\mathalpha}{letters}{"16}
\DeclareMathSymbol{\nu}{\mathalpha}{letters}{"17}
\DeclareMathSymbol{\xi}{\mathalpha}{letters}{"18}
\DeclareMathSymbol{\pi}{\mathalpha}{letters}{"19}
\DeclareMathSymbol{\rho}{\mathalpha}{letters}{"1A}
\DeclareMathSymbol{\sigma}{\mathalpha}{letters}{"1B}
\DeclareMathSymbol{\tau}{\mathalpha}{letters}{"1C}
\DeclareMathSymbol{\upsilon}{\mathalpha}{letters}{"1D}
\DeclareMathSymbol{\phi}{\mathalpha}{letters}{"1E}
\DeclareMathSymbol{\chi}{\mathalpha}{letters}{"1F}
\DeclareMathSymbol{\psi}{\mathalpha}{letters}{"20}
\DeclareMathSymbol{\omega}{\mathalpha}{letters}{"21}
\DeclareMathSymbol{\varepsilon}{\mathalpha}{letters}{"22}
\DeclareMathSymbol{\vartheta}{\mathalpha}{letters}{"23}
\DeclareMathSymbol{\varpi}{\mathalpha}{letters}{"24}
\DeclareMathSymbol{\varrho}{\mathalpha}{letters}{"25}
\DeclareMathSymbol{\varsigma}{\mathalpha}{letters}{"26}
\DeclareMathSymbol{\varphi}{\mathalpha}{letters}{"27}

\RedeclareMathOperator{\div}{\textbf{div}}

\renewcommand{\epsilon}{\varepsilon}

\renewcommand{\d}{{\rm d}}

\newcommand{\order}{\ensuremath{\mathcal{O}}}

\newcommand{\half}{\ensuremath{\frac{1}{2}}}

\newcommand{\del}{\partial}

\newcommand{\rec}{\ensuremath{\text{rec}}}

\newcommand{\x}{\ensuremath{\vec x}\xspace}

\newcommand{\vel}{\ensuremath{\vec v}\xspace}

\newcommand\ie{i.e.\ }
\newcommand\eg{e.g.\ }

\let\origdoublepage\cleardoublepage
\renewcommand{\cleardoublepage}{%
\clearpage{\pagestyle{empty}\origdoublepage}}

\makeatletter
\newcommand{\ChapterOutsidePart}{%
\def\toclevel@chapter{-1}\def\toclevel@section{0}\def\toclevel@subsection{1}}
\newcommand{\ChapterInsidePart}{%
\def\toclevel@chapter{0}\def\toclevel@section{1}\def\toclevel@subsection{2}}
\makeatother

\newcommand{\dt}{\partial_t}
\newcommand{\tdt}{\frac{d}{dt}}
\newcommand{\dx}{\partial_x}
\newcommand{\dy}{\partial_y}

\newcommand{\q}{\vec q}
\newcommand{\w}{\vec w}
\newcommand{\s}{\vec s}
\newcommand{\f}{\vec f}
\newcommand{\g}{\vec g}

\newcommand{\Q}{\vec Q}
\newcommand{\W}{\vec W}
\renewcommand{\S}{\vec S}
\newcommand{\F}{\vec F}

\newcommand{\gx}{g_x}
\newcommand{\gy}{g_y}
\newcommand{\eq}{\text{eq}}

\newcommand{\dwb}{DWB\xspace}
\newcommand{\la}{LA\xspace}
\newcommand{\laeos}{LA\xspace}
\newcommand{\laeosf}{LA-S\xspace}
\newcommand{\dwbeos}{DWB\xspace}
\newcommand{\dwbeosf}{DWB-S\xspace}

\newcommand{\tdwb}{discretely well-balanced method\xspace}

\newcommand{\tla}{local approximation method\xspace}

\begin{document}


\begin{frontmatter}

\title{%
	High order discretely well-balanced method for arbitrary hydrostatic atmospheres
}
\author[<1>]{Jonas P.\ Berberich\corref{cor1}}
\ead{jonas.berberich@mathematik.uni-wuerzburg.de}
\author[<2>]{Roger K\"appeli} 
\ead{roger.kaeppeli@sam.math.ethz.ch}
\author[<3>]{Praveen Chandrashekar} 
\ead{praveen@tifrbng.res.in}
\author[<1>]{Christian Klingenberg}
\cortext[cor1]{Corresponding author: 
Tel.: +49 931 31-88861;  
Fax: +49 931 31-83494;}
\ead{klingen@mathematik.uni-wuerzburg.de}
\address[<1>]{Dept.~of Mathematics, Univ.~of W\"urzburg, Emil-Fischer-Stra{\ss}e 40, 97074 W\"urzburg, Germany}
\address[<2>]{Seminar for Applied Mathematics (SAM),
            Department of Mathematics,
            ETH Z\"urich, CH-8092 Z\"urich, Switzerland}
\address[<3>]{TIFR Center for Applicable Mathematics, Bengaluru, Karnataka 560065, India}


\begin{abstract}
We introduce novel high order well-balanced finite volume methods for the full compressible Euler system with gravity source term. 
They require no \`a priori knowledge of the hydrostatic solution which is to be well-balanced and are not restricted to certain classes of hydrostatic solutions.
In one spatial dimension we construct a method that exactly balances a high order discretization of any hydrostatic state. The method is extended to two spatial dimensions using a local high order approximation of a hydrostatic state in each cell. The proposed simple, flexible, and robust methods are not restricted to a specific equation of state. Numerical tests verify that the proposed method improves the capability to accurately resolve small perturbations on hydrostatic states.
\end{abstract}


\end{frontmatter}



\section{Introduction}
\label{sec:intro}

In many applications, the compressible Euler equations arise as a model for flow of inviscid compressible fluids such as air. Finite volume methods are commonly utilized to numerically approximate solutions of this system since they are conservative and capable of resolving shocks by construction. Fluid dynamics in atmospheres can be modeled by adding a gravity source term to the Euler system. This model admits non-trivial static, i.e., time independent solutions, the \emph{hydrostatic solutions}. They are described by the \emph{hydrostatic equation}
\begin{equation}
\label{eq:hystat}
\vel=0, \qquad	\nabla p = \rho\g
\end{equation}
which models the balance between the gravity $\rho \g$, where $\rho$ is the gas \emph{density} and $\g$ is the \emph{gravitational acceleration}, and the pressure gradient $\nabla p$, where $p$ is the gas \emph{pressure}. Additionally, the solution must satisfy the constitutive relation between pressure, density, and internal energy density $\varepsilon$. This relation is called \emph{equation of state} (EoS) and it has to be added to the Euler system to close it. In many practical simulations, the dynamics are considered which are close to a hydrostatic state. Standard finite volume methods usually introduce truncation errors to hydrostatic states, which can be larger than the actual perturbations related to the simulated dynamical process. Hence, the small-scale dynamics can only be resolved on very fine grids, which leads to high computational cost. This creates the demand for so-called \emph{well-balanced} methods which are constructed to be free of a truncation error at hydrostatic states.

The idea of well-balanced methods is very common especially for the shallow water equations with bottom topography. The hydrostatic solution for the shallow water equations, the so-called lake-at-rest solution, can be given in the form of an algebraic relation. This favors the construction of well-balanced methods since the algebraic relation can be used to perform a local hydrostatic reconstruction, which is the main tool to construct well-balanced methods. Examples can be found in \cite{Bermudez1994,LeVeque1998b,Audusse2004,Noelle2006,Noelle2007,Xing2011,Castro2018} and references therein. Also, for the Ripa model, which is closely related to shallow water model, there are well-balanced methods (\eg \cite{Touma2015,Desveaux2016b} and references therein).

For the compressible Euler equations with gravity source term, the situation is more complicated, since the hydrostatic states are not given by an algebraic relation but by a differential equation (\cref{eq:hystat}) together with an EoS. 
Especially complicated EoS can increase the difficulty of performing a local hydrostatic reconstruction. The result is that there do not exist methods which are well-balanced for all EoS and all types of hydrostatic solutions, and all existing methods so far known for Euler equations with gravity bear some restriction. \refb{We can classify well-balanced methods broadly into three types, which may help in understanding their differences and limitations and their domain of usefulness.}

\refb{%
In the first type of well-balancing approaches, \`a priori knowledge of the hydrostatic state which is to be well-balanced is assumed.
	\begin{definition}[Type 1]
		\label{def:wb_type1}
		A numerical method is well-balanced (type 1) if it exactly preserves any hydrostatic state that is given as analytical formula or in terms of discrete data on the grid.
	\end{definition}
	This allows the methods to be general, such that they can balance arbitrary hydrostatic states to arbitrary EoS~\cite{Berberich2016,Thomann2019,Berberich2019,Kanbar2020}. High order methods of this type are given in~\cite{Castro2020} for one spatial dimension and in~\cite{Klingenberg2019} for two spatial dimensions. The most general well-balanced method of this type is presented in~\cite{Berberich2020}. It can be applied to balance any stationary solution and follow any time-dependent solution exactly in a high order method. It can be utilized for any multi-dimensional hyperbolic balance law. In many applications the hydrostatic solution which is of interest is known \`a priori, either analytically or as discrete data.

The second type of methods is developed to well-balance certain classes of hydrostatic solutions. 
	\begin{definition}[Type 2]
		\label{def:wb_type2}
		A numerical method is well-balanced (type 2) if it exactly preserves hydrostatic states that satisfy a certain barotropic condition. 
	\end{definition}
	Examples for these barotropic conditions are constant temperature or entropy or a polytropic relation between density and pressure.
Other hydrostatic states still lead to truncation errors. Well-balanced methods of type 2 are given in \cite{Cargo1994,LeVeque1999,LeVeque2011,Desveaux2014,Kaeppeli2014,Chandrashekar2015,Desveaux2016,Ghosh2016,Touma2016} and there are also higher order methods~\cite{Xing2013,Grosheintz2019}. The restriction to certain classes of hydrostatic solutions usually also manifests in a restriction to the EoS, which is in most cases the ideal gas EoS. The presence of such additional condition provides an algebraic relation which can be used to perform hydrostatic reconstruction.

	However, there are also situations, in which a more complicated structure can be expected for the hydrostatic solution and no \`a priori knowledge can be assumed. If the relevant hydrostatic state is not known and cannot be expected to be of one of the classes balanced in the first type of methods (for example because a complicated non-ideal gas EoS is used), a method of the third type  can be useful.
		\begin{definition}[Type 3 - discretely well-balanced methods]
			\label{def:wb_type3}
			A numerical method is well-balanced (type 3) if it exactly preserves a discrete approximation of a hydrostatic state without additional assumptions and \`a priori knowledge of the hydrostatic solution. We then refer to the method as \emph{discretely well-balanced} method.
		\end{definition}
		These methods are based on balancing some approximation to a hydrostatic state. The second order methods introduced in~\cite{Kaeppeli2014,Kaeppeli2016,Varma2019}, for example, are exactly well-balanced for certain classes of hydrostatic states; otherwise they balance a second order discretization to any hydrostatic state.
		A high order method of this type has been presented in~\cite{Franck2016} in the framework of Lagrangian+remap schemes for the ideal gas EoS. This article also formally introduces the concept of discretely well-balanced methods.
		There are also methods, which balance a global approximation to any hydrostatic state \cite{Cargo1994,Chertock2018a}, rather than a local one.
	}
		
		Note that this classification into three different approaches is not strict. Some methods of the third type, for example, can also be seen to be of the \refb{second} type, since they also might balance certain classes of solutions exactly~\cite{Kaeppeli2014,Varma2019}. \refb{Recently, methods have been developed that are able to balance stationary solutions of the Euler equations with non-vanishing velocity~\cite{Gaburro2018b,Berberich2020,Grosheintz2020}. In this article, however, we only consider hydrostatic states, i.e., stationary solutions with zero velocity.}
		
		In the present article, we develop a method of type 3 following the approach of balancing a local approximation to any hydrostatic state. A high order accurate local discrete approximation to a hydrostatic state is constructed based on the local distribution of density and gravitational acceleration. Well-balancing is then achieved using hydrostatic reconstruction in the manner of~\cite{Audusse2004} and a suitable source term discretization. The construction of our method allows for general EoS. For one spatial dimension it can be shown that a high order discretization of any hydrostatic state can be well-balanced exactly. The method is extended to two spatial dimensions using a local approximation to the hydrostatic state in each cell. Numerical tests validate a significant increase of accuracy on small perturbations to hydrostatic states. An increased order of accuracy in the convergence to exact hydrostatic states is observed in one and two spatial dimensions. The method allows a free choice of the components of a Runge--Kutta finite volume scheme, including reconstruction method, numerical flux function, and ODE solver. The only restriction is that the numerical flux function has to satisfy the contact preservation property. The methods can be implemented as simple modifications to existing Runge--Kutta finite volume codes.
		
		In this paper, we extend the one-dimensional well-balanced finite volume
		schemes \cite{Kaeppeli2015,Kaeppeli2016}
		beyond second-order accuracy. The proposed scheme possesses the following novel set of features and properties:
		\begin{enumerate}
			\item[(i)] An arbitrary high-order accurate one-dimensional local hydrostatic
			profile is constructed without any explicit assumption on the thermal
			stratification such as the entropy, temperature or any other barotropic
			relation.
			\item[(ii)] An arbitrarily high-order accurate one-dimensional equilibrium
			preserving reconstruction based on the local equilibrium profile is
			built.
			\item[(iii)] A source term discretization is designed such that the scheme is
			well-balanced for gravitational forces aligned with one computational
			axis, i.e., the scheme preserves a discrete high-order approximation of
			one-dimensional hydrostatic equilibrium.
			\item[(iv)] The scheme requires only standard components: a high-order reconstruction
			procedure and a consistent and Lipschitz continuous numerical flux
			function capable of resolving stationary contact discontinuities
			exactly.
			Hence, it can be implemented with ease within any standard high-order
			finite volume code.
			\item[(v)] The scheme can handle any form of equation of state including tabulated,
			which is particularly important in astrophysical applications.
			\item[(vi)] Although the scheme is not well-balanced for general multi-dimensional
			hydrostatic configurations, the presented numerical experiments
			demonstrate the nevertheless vastly increased resolution capabilities.
		\end{enumerate}

		The rest of the article is structured as follows.
		In \cref{subsec:nm_1d}, the one-dimensional compressible Euler equations with gravity source term are introduced. A standard high order finite volume method for these is revised in \cref{subsubsec:stan}. In \cref{sec:discrete_wb} we develop the \tdwb (\dwb) for an ideal gas law and add the description for general EoS in \cref{sec:discrete_wb_eos}. The well-balanced property and high order accuracy are stated in \cref{thm:nm_1d}. The method is modified in \cref{sec:compact_method} to reduce the stencil and the \tla (\la) is obtained.  In \cref{sec:ne}, numerical experiments in one spatial dimension are conducted. It is verified numerically that the \dwb method is exactly well-balanced on the discrete approximation to a hydrostatic state and \dwb and \la methods are numerically shown to converge towards exact hydrostatic states with an increased order of convergence  (\cref{sec:1d_isothermal}). The capability of the methods to accurately resolve small perturbations on hydrostatic states is illustrated in \cref{sec:1d_isothermal_pert}. Tests with a non-ideal gas EoS are presented in \cref{sec:1d_non_ideal}. In \cref{sec:robustness} we verify the robustness of the methods in the presence of discontinuities.
		\Cref{sec:2d} is dedicated to the extension to two spatial dimensions. In \cref{sec:eul_2d} we introduce the two-dimensional Euler equations with gravity source term. Subsequently, the \la method is extended to two spatial dimensions (\cref{sec:wb_2d}). In \cref{sec:ne_2d}, numerical tests of the \la method in two spatial dimensions are presented. On a two-dimensional polytrope, which is a hydrostatic state, the increased order of accuracy is observed also for two spatial dimensions (\cref{sec:test_polytrope}). A perturbation is added to the polytrope in \cref{sec:test_polytrope_pert}. Rayleigh--Taylor instabilities on a radial setup are simulated in \cref{sec:rrt} and the increased accuracy of the \la method is shown. In \cref{sec:conclusions} we close the article with some conclusions and an outlook.

		\section{One-dimensional finite volume methods}\label{sec:nm}
		\subsection{One-dimensional compressible Euler equations with gravity}\label{subsec:nm_1d}
		We consider the one-dimensional compressible Euler equations with gravitation
		in Cartesian coordinates and write them in the following compact form
		\begin{equation}
		\label{eq:nm_1d_0010}
		\del_t\q + \del_x\f = \s
		,
		\end{equation}
		where
		\begin{align}
		\label{eq:nm_1d_0020}
		\q = \begin{pmatrix}
		\rho   \\
		\rho u \\
		E
		\end{pmatrix}
		, \quad
		\f(\q) = \begin{pmatrix}
		\rho u     \\
		\rho u^2+p \\
		u(E+p)
		\end{pmatrix}
		\quad \text{and} \quad
		\s(\q,g) =  \begin{pmatrix}
		0        \\
		\rho g   \\
		\rho u g
		\end{pmatrix}
		\end{align}
		are the vectors of conserved variables, fluxes and source terms,
		respectively.
		Moreover, the total energy is given by $E = \varepsilon + \rho u^2/2$ and we denote the primitive variables by $\w=[\rho,u,p]^T$.
		The equation of state (EoS) closes the system by relating the pressure $p$
		to the mass density $\rho$ and internal energy density $\varepsilon$,
		\ie $p = p(\rho,\varepsilon)$.
		A simple EoS is provided by the ideal gas law
		\begin{equation}
		\label{eq:nm_1d_0030}
		p = (\gamma - 1) \varepsilon
		,
		\end{equation}
		where $\gamma$ is the ratio of specific heats.
		However, we stress that the well-balanced schemes elaborated below are not
		restricted to any particular EoS, which is important especially in
		astrophysical applications.
		\subsection{Standard high-order finite volume methods}
		\label{subsubsec:stan}
		In this section, we briefly describe an high order accurate standard high-order finite volume scheme in order
		to fix the notation and set the stage for the detailed presentation of our
		novel well-balanced schemes. The spatial domain of interest $\Omega$ is discretized into a finite number
		$N$ of cells or finite volumes
		$\Omega_{i} = \left[x_{i-\half} ,x_{i+\half}\right]$, $i=1,\ldots,N$. \
		The $x_{i\mp\half}$ denote the left/right cell interfaces and the point
		$x_i = \left(x_{i-\half} + x_{i+\half}\right)/2$ is the cell center.
		
		A semi-discrete finite volume scheme is then obtained by integrating
		\cref{eq:nm_1d_0010} over a cell $\Omega_{i}$
		\begin{equation}
		\label{eq:nm_1d_0040}
		\dt\hat\Q_i(t)
		= \mathcal{L}\left(\hat\Q_i\right)
		= - \frac{1}{\Delta x} \left[ \F_{i+\half} -\F_{i-\half}\right]
		+ \hat\S_{i}
		.
		\end{equation}
		Here $\hat\Q_{i}$ denotes the approximate cell average of the conserved
		variables in cell $\Omega_{i}$ at time $t$
		\begin{equation}
		\label{eq:nm_1d_0050}
		\hat\Q_i(t) \approx \hat\q_i(t)
		= \frac{1}{\Delta x} \int_{\Omega_{i}} \q(x,t) ~ dx
		,
		\end{equation}
		where $\hat\q_i(t)$ denotes the cell average of the exact solution $\q(x,t)$.
		In the following, a quantity with a hat $\hat\cdot$ indicates a cell average
		and one without a hat indicates a point value.
		Note that this distinction is essential for higher-order methods of order
		greater than two.
		Likewise, $\hat\S_{i}$ denotes the approximate cell average of the source
		terms at time $t$
		\begin{equation}
		\label{eq:nm_1d_0060}
		\hat\S_i(t) \approx \hat\s_i(t)
		= \frac{1}{\Delta x}
		\int_{\Omega_{i}} \s\left(\q(x,t),g(x)\right) ~ dx
		,
		\end{equation}
		where $\hat\s_i(t)$ denotes the cell average of the exact source terms
		$\s(\q,g)$.
		\paragraph{Reconstruction}
		As the basic unknowns in the finite volume method are cell average values, we need some reconstruction procedure to recover the detailed spatial variation of the solution in order to obtain high order accuracy. 
		We denote a reconstruction procedure that recovers an $m$-th ($m$ odd) order accurate
		point value of the conserved variables at location $x$ within cell $\Omega_{i}$
		from the cell averages by
		\begin{equation}
		\label{eq:nm_1d_0071}
		\Q_{i}^\text{rec}(x) = \mathcal{R}\left( x;
		\left\{\hat\Q_{j}\right\}_{j \in \mathcal S_{i}}
		\right)
		.
		\end{equation}
		where 
		$\mathcal S_{i} = \left\{i-\frac{m-1}2,\dots,i,\dots,i+\frac{m-1}2 \right\}$ is the stencil of the
		reconstruction. 
		
		Many such reconstruction procedures have been developed in the literature,
		and a non-exhaustive list includes the Total Variation Diminishing (TVD) and
		the Monotonic Upwind Scheme for Conservation Laws (MUSCL) methods
		~\cite{VanLeer1979,Harten1983b,Sweby1984,Laney1998,LeVeque2002,Toro2009},
		the Piecewise Parabolic Method (PPM)~\cite{Colella1984},
		the Essentially Non-Oscillatory (ENO)~\cite{Harten1987},
		Weighted ENO (WENO) (see~\cite{Shu2009} and references therein) and
		Central WENO (CWENO) methods~\cite{Cravero2018}.
		
		In this paper, we will use the $m$-th order CWENO reconstructions~\cite{Cravero2018} with $m$ being an odd integer. A distinctive feature of CWENO reconstruction is that it provides an analytic expression for the reconstruction polynomial and not just a value at some particular point within the cell. This is particularly useful for solving balance laws which require the source terms to be integrated over each cell.  In particular, for the numerical tests, we will use the third-order accurate CWENO3 from~\cite{Kolb2014} and the fifth-order accurate CWENO5
		from~\cite{Capdeville2008}.
		However, the schemes developed in this paper are not restricted to this
		choice of reconstruction schemes as long as the reconstruction yields a function available in the whole cell and extendable to neighboring cells.
		\paragraph{Numerical flux}
		At each cell face, we have obtained two solution values from the reconstructions in the two surrounding cells. The flux across the $i+\half$ face is obtained from a numerical flux formula $\mathcal{F}$,
		\begin{equation}
		\label{eq:nm_1d_0080}
		\F_{i+\half} = \mathcal{F}\left(\Q^\text{rec}_{i}\left(x_{i+\half}\right),\Q^\text{rec}_{i+1}\left(x_{i+\half}\right)\right)
		,
		\end{equation}
		which is usually based on (approximately) solving Riemann problems at cell interfaces.
		The numerical flux function $\mathcal{F}$ is required to be consistent,
		i.e.\ $\mathcal{F}(\Q,\Q) = \f(\Q)$ and Lipschitz continuous. Moreover, we will require that it satisfies the \emph{contact property}.
		\begin{definition}[Contact property]
			\label{def:contact_property}
			Let $\rho_{L}$ ($\rho_{R}$) be the density on the left (right) side of a
			contact discontinuity and $p$ the constant pressure.
			A numerical flux function $\mathcal{F}$ for the one-dimensional Euler
			equations that satisfies the condition
			\begin{equation}
			\mathcal{F}\left(  \Q(\rho_{L}, 0, p)
			, \Q(\rho_{R}, 0, p)
			\right)
			=
			\left[ 0, p, 0 \right]^{T}
			\label{eq:contact_property}
			\end{equation}
			is said to have the \emph{contact property}. Here, $\Q = \Q(\rho,u,p)$ denotes the transformation from primitive to conserved variables.
		\end{definition}
		This property ensures the ability of a numerical flux to exactly capture
		stationary contact discontinuities of the Euler equations.
		In our tests below, we will use the well-known approximate Riemann solver by
		Roe~\cite{Roe1981}.
		Another well-known flux with the contact property is the HLLC
		flux~\cite{Toro1994}.
		The Rusanov flux is not able to capture contact discontinuities since it does
		not satisfy \cref{def:contact_property}. Note, that the contact property is, besides consistency, the only requirement for a numerical flux to use with the well-balanced methods proposed in this article. This gives a lot of freedom, since there are many contact property satisfying numerical fluxes with different properties available in literature. This includes for example entropy stable fluxes (\eg \cite{Ismail2009,Chandrashekar2013,Berberich2020b}) and numerical fluxes suitable for the simulation of low Mach number flows (\eg \cite{Turkel1987,Li2008,Li2013,Miczek2015,Barsukow2017,Berberich2020b})
		\paragraph{Source term discretization}
		We assume that the gravitational acceleration is given and can be
		evaluated anywhere in the computational domain as needed by the scheme.
		
		If the gravitational acceleration $g(x)$ is a given function, it can be
		evaluated directly at the quadrature nodes.
		If the gravitational acceleration $g(x)$ is only known at discrete points,
		then a suitable interpolation can be used.
		Note that the interpolation needs to be at least as accurate as the desired
		design order of the scheme.
		In the numerical examples below, and also in the construction of our well-balanced method, we assume that point values $g_{i}$ of gravitational
		acceleration are given at all cell centers $x_i$.
		These point values are then interpolated 
		using sufficiently high order polynomial interpolation. 
		In an $m$-th order accurate method we define
		\begin{equation}
		\label{def:gint_1d}
		g_i^\text{int}(x)=\sum_{k=0}^{m-1}a_k(x-x_i)^k
		\qquad\text{for}\qquad
		x\in\Omega_{i},
		\end{equation}
		where the coefficients $a_0,\dots,a_{m-1}\in \mathbb R$ are defined via demanding
		\begin{equation}
		\label{def:gint_1d_coeffs}
		g_i^\text{int}(x)=g_j
		\qquad\text{for}\qquad
		j\in\left\{ i-\frac{m-1}2,\dots,i+\frac{m-1}2 \right\}.
		\end{equation}
		If the gravitational acceleration is not smooth enough, a CWENO interpolation can be applied instead.
		An accurate approximation of the cell average of the source terms is then obtained
		by integration:
		\begin{equation}
		\label{eq:nm_1d_0100}
		\hat\S_{i}(t) 
		=\frac{1}{\Delta x}
		\int_{\Omega_{i}}
		\begin{pmatrix}
		0\\
		\s_i(x)\\
		\frac{(\rho u)_i^{\rm rec}(x)}{\rho_i^{\rm rec}(x)}\s_i(x)
		\end{pmatrix}
		~ dx
		\qquad\text{with}\qquad
		\s_i(x)=\rho_i^\rec(x)g_i^{\rm int}(x),
		\end{equation}
		where $\rho_i^\rec$,and $(\rho u)_i^\rec$ are obtained from the
		reconstruction procedure \cref{eq:nm_1d_0071}.
		In this article we use exact integration which is possible because of the way we obtain $g^\text{int}$. In general, however, exact integration is not always possible. In that case a sufficiently accurate quadrature rule has to be applied.
		%
		\paragraph{Time-stepping}
		For the temporal integration, the time domain of interest $[0,T]$ is
		discretized into time steps $\Delta t = t^{n+1} - t^{n}$, where the
		superscripts label the different time levels.
		The semi-discrete scheme \cref{eq:nm_1d_0040} is evolved in time using some
		ODE integrator.
		For this purpose, we apply explicit Runge--Kutta methods.
		To achieve third-order accuracy in time, we use a third-order accurate, four
		stage explicit Runge--Kutta method \cite{Kraaijevanger1991}.
		To achieve fifth-order accuracy in time, we apply a fifth-order accurate
		Runge--Kutta method (the standard-method from \cite{Rabiei2012}).
		Furthermore, the time step $\Delta t$ is required to fulfill a CFL condition
		for stability.
		
		This concludes the description of a standard high-order accurate finite volume scheme.
		We refer to the many excellent available textbooks for further information
		and detailed derivation, see~\cite{Laney1998,LeVeque2002,Hirsch2007,Toro2009}.
		However, a standard finite volume method is in general not able to exactly
		balance hydrostatic equilibrium solutions.
		Next, we present the necessary modifications to the reconstruction procedure,
		which render this precise balance possible, and result in a family of
		high-order well-balanced schemes able to preserve a high-order accurate
		discrete form of the equilibrium.
		
		\subsection{The \tdwb}
		\label{sec:discrete_wb}
		The main idea of achieving well-balancing lies in a carefully designed reconstruction process in two steps.
		In the first step, a local equilibrium profile is determined within each cell
		that is consistent up to the desired order of accuracy with the cell-averaged
		conserved variables.
		We emphasize that this step is the principal novelty of the presented
		schemes.
		It generalizes the second-order schemes \cite{Kaeppeli2015,Kaeppeli2016},
		able to preserve discrete hydrostatic states without any explicit assumption
		of thermal equilibrium, to arbitrary orders of accuracy.
		In the second step, the cell's equilibrium profile is extrapolated to its
		neighboring cells and a high-order reconstruction of the equilibrium
		perturbation is performed.
		

		
		\paragraph{Step 1}  We begin by the construction of the local high-order equilibrium profile within the $i$-th cell
		\begin{equation}
		\label{eq:nm_1d_wb_0010}
		\Q_{i}^{\eq}(x) = \begin{pmatrix}
		\rho_{i}^{\eq}(x)      \\
		0                      \\
		\varepsilon_{i}^{\eq}(x)
		\end{pmatrix}
		\end{equation}
		fulfilling
		\begin{equation}
		\label{eq:nm_1d_wb_0020}
		\frac{\d p_{i}^{\eq}}{\d x}
		= s_{i}(x)
		.
		\end{equation}
		We define the local equilibrium source term related to the $i$-th cell as 
		\begin{equation}
		\label{def:source_term_dwb}
		s_{i}(x) = \sum_{k\in\mathcal S_i} \rho_{k}^{\eq}(x)g_i^\text{int}(x) ~ \mathbbm{1}_{\Omega_{k}}(x) \qquad\textrm{and}\qquad
		\mathbbm{1}_{\Omega_{k}}(x) = \begin{cases}
		1 \text{ if } x \in     \Omega_{k} \\
		0 \text{ if } x  \notin \Omega_{k}
		\end{cases},
		\end{equation}
		where $\mathcal S_i$ is the stencil of the reconstruction $\mathcal R$.
		The equilibrium mass density $\rho_{i}^{\eq}(x)$, internal energy density
		$\varepsilon_{i}^{\eq}(x)$ and the pressure $p_{i}^{\eq}(x)$ are related
		through the EoS
		\begin{equation}
		\label{eq:nm_1d_wb_0030}
		\varepsilon^\eq_{i}(x) = \varepsilon \left( \rho_{i}^{\eq}(x)
		, p_{i}^\eq(x) \right)
		.
		\end{equation}
		As a matter of fact, the hydrostatic equilibrium stratification
		\cref{eq:nm_1d_wb_0010} is not uniquely specified by
		\cref{eq:nm_1d_wb_0020,eq:nm_1d_wb_0030} (indeed, we have three physical
		quantities and only two equations linking them).
		To fully determine the equilibrium, one additional relation is needed.
		For that purpose, we demand that the local equilibrium density profile
		$\rho_{i}^{\eq}(x)$ corresponds to the density profile obtained from the
		standard reconstruction procedure \cref{eq:nm_1d_0071}
		\begin{equation}
		\label{eq:nm_1d_wb_0040}
		\rho_{i}^{\eq}(x) = \rho_{i}^{\text{rec}}(x)
		.
		\end{equation}
		Note that this is the major difference with well-balanced schemes assuming an
		explicit thermal equilibrium, e.g., isothermal, isentropic, polytopic or, in
		general, barotropic (density is a function of pressure only) conditions. The density profile can be arbitrary and we do not impose any further restriction on the structure of the hydrostatic solution.
		
		Now, we construct the equilibrium pressure profile within the $i$-th cell by simply
		integrating \cref{eq:nm_1d_wb_0020} as
		\begin{equation}
		\label{eq:nm_1d_wb_0050}
		p_{i}^{\eq}(x) = p_{0,i}
		+ \int_{x_{i}}^{x} s_{i}(\xi) ~ \d \xi
		,
		\end{equation}
		where $p_{0,i}$ is the point value of the pressure anchoring the equilibrium
		profile at cell center $x_{i}$.
		Note that the above integral can be evaluated analytically for any
		$x \in \Omega_{i}$ since the equilibrium density $\rho_{i}^{\eq}(x)$ (\cref{eq:nm_1d_wb_0040}) and the interpolated gravitational acceleration $g_i^\text{int}$ are
		simply polynomials.
		
		It remains to fix the equilibrium pressure $p_{0,i}$ at cell center.
		To this end, we require that the equilibrium profile $\Q_{i}^{\eq}(x)$
		matches the cell-averaged conserved variables $\hat\Q_{i}$ in the cell
		$\Omega_{i}$ up to the desired order of accuracy.
		For the equilibrium density, this is fulfilled by construction (since it is
		identical to the profile from the standard reconstruction procedure).
		For the internal energy density, the requirement leads to the following
		equation
		\begin{equation}
		\label{eq:nm_1d_wb_0060}
		\hat\varepsilon_{i}
		= \frac{1}{\Delta x}
		\mathcal{Q}_{i} \left( \varepsilon (\rho_{i}^{\eq}, p_{i}^{\eq}) \right)
		= \frac{1}{\Delta x}
		\sum_{\alpha=1}^{N_{q}}
		\omega_{\alpha}
		~
		\varepsilon (\rho_{i}^{\eq}(x_{i,\alpha}), p_{i}^{\eq}(x_{i,\alpha}))
		,
		\end{equation}
		where $\mathcal{Q}_{i}$ is a $q$-th order accurate
		quadrature rule with $\sum_{\alpha} \omega_{\alpha} = \Delta x$.
		Moreover, $\hat\varepsilon_{i}$ on the left-hand side is an estimate of the
		cell-averaged internal energy density in cell $\Omega_{i}$.
		\refb{%
		We estimate $\hat\varepsilon_{i}$ directly
		from the cell-averaged conserved variables as
		\begin{equation}
		\label{eq:nm_1d_wb_0070}
		\hat{\varepsilon}_{i} = \hat{E}_{i}
		- \frac{1}{2\Delta x} \mathcal Q_i\left(\frac{\left((\rho u)_{i}^\rec\right)^{2}}{{\rho}_{i}^\rec}\right)
		,
		\end{equation}
		where $(\rho u)_i^\rec$ has been obtained via the standard reconstruction procedure.
		}
		
		For readability we restrict to an ideal gas EoS in the following. The procedure for general EoS in described in \cref{sec:discrete_wb_eos}.
		For an ideal gas law, \cref{eq:nm_1d_wb_0060} can be solved analytically
		\begin{equation}
		\label{eq:nm_1d_wb_0080}
		\begin{aligned}
		p_{0,i} 
		& = (\gamma - 1) \hat{\varepsilon}_{i}
		- \frac{1}{\Delta x}
		\mathcal{Q}_{i} \left(
		\int_{x_{i}}^{x}
		s_{i}(\xi) ~ \d \xi
		\right) \\
		& = (\gamma - 1) \hat{\varepsilon}_{i}
		- \frac{1}{\Delta x}
		\sum_{\alpha=1}^{N_{q}}
		\omega_{\alpha}
		\int_{x_{i}}^{x_{i,\alpha}}
		s_{i}(\xi) ~ \d \xi
		,
		\end{aligned}
		\end{equation}
		where we used the ideal gas EoS (\cref{eq:nm_1d_0030}) and that for the quadrature weights we have $\sum_{\alpha} \omega_{\alpha} = \Delta x$.
		Again, since the equilibrium density $\rho_{i}^{\eq}(x)$ and the interpolated gravitational acceleration $g_i^{\rm int}$ are polynomials, the integral can be
		evaluated analytically in a straightforward manner.
		Now that the pressure at cell center $p_{0,i}$ is fixed, we have completely
		specified the high-order accurate local representation of the equilibrium
		conserved variables \cref{eq:nm_1d_wb_0010} within cell $\Omega_{i}$.

		\paragraph{Step 2}
		Next, we develop the high-order equilibrium preserving reconstruction
		procedure.
		To this end, we follow the existing methodology (see the references in the
		\cref{sec:intro}) and decompose in every cell the solution into an equilibrium
		and a (not necessarily small) perturbation part.
		The equilibrium part in cell $\Omega_{i}$ is simply given by
		$\Q_{i}^\eq(x)$ of Eq. \eqref{eq:nm_1d_wb_0010}.
		The perturbation part in cell $\Omega_{i}$ is obtained by applying the
		standard reconstruction procedure $\mathcal{R}$ to the cell-averaged
		equilibrium perturbation
		\begin{equation}
		\label{eq:nm_1d_wb_0110}
		\delta \Q_{i}^{\text{rec}}(x)
		= \mathcal{R}
		\left(
		x
		;
		\left\{ \hat\Q_{k}
		- \frac1{\Delta x}\mathcal Q_k\left(\Q_{i}^\eq\right)\right\}_{k \in \mathcal S_{i}}
		\right)
		\end{equation}
		We note that the cell average of the equilibrium perturbation in cell
		$\Omega_{k}$ is obtained by taking the difference between the cell average
		$\hat\Q_{k}$ and the cell average of the local high order equilibrium profile $\hat\Q^\eq_{i}$
		in cell
		$\Omega_{k}$.
		This construction results in a $\min(q,r)$-th order accurate representation of the equilibrium perturbation in cell $\Omega_{i}$.
		
		The complete equilibrium preserving reconstruction $\mathcal{W}$ is obtained by
		the sum of the equilibrium and perturbation reconstruction
		\begin{equation}
		\label{eq:nm_1d_wb_0120}
		\Q_{i}^{\text{rec}}(x) = \mathcal{W} \left( x ; \{  \hat\Q_{k} \}_{k \in \mathcal S_{i}} \right)
		= \Q_{i}^\eq(x) + \delta \Q_{i}^{\text{rec}}(x)
		.
		\end{equation}
		By construction, this reconstruction will preserve any equilibrium of the form
		\cref{eq:nm_1d_wb_0020} since the perturbation $\delta \Q_{i}^{\text{rec}}(x)$
		vanishes under these conditions.
		
		This concludes the description of the equilibrium preserving reconstruction
		procedure.
		Replacing only this component in a standard finite volume method renders it
		well-balanced for the above discrete form of arbitrary hydrostatic equilibrium,
		i.e., only a mechanical equilibrium and no thermal equilibrium needs to be
		explicitly assumed.
		In the rest of the article, we will refer to the method introduced in this
		section as discretely well-balanced (\dwb) method.

		\subsubsection{The \tdwb for arbitrary EoS}
		\label{sec:discrete_wb_eos}
		
		In this section, we present the details for using the developed well-balanced
		schemes with a general EoS.
		In that case, \cref{eq:nm_1d_wb_0060} is not explicitly solvable for the
		cell center equilibrium pressure $p_{0,i}$ in cell $\Omega_{i}$, which can be rewritten as
		\begin{equation}
		\label{eq:appendix_wb_rec_0010}
		f(p_{0,i}) = 0
		,
		\end{equation}
		where
		\begin{equation}
		\label{eq:appendix_wb_rec_0020}
		f(p)
		=   \hat\varepsilon_{i}
		- \frac{1}{\Delta x}
		\sum_{\alpha=1}^{N_{q}}
		\omega_{\alpha}
		~
		\varepsilon \left(  \rho_{i}^{\eq}(x_{i,\alpha})
		, ~ p
		+ \int_{x_{i}}^{x_{i,\alpha}}
		s_{i}(\xi)
		~ \d \xi
		\right)
		.
		\end{equation}
		Here, $\hat\varepsilon_{i}$ is the estimate of the cell-averaged internal
		energy density \cref{eq:nm_1d_wb_0070} and $\mathcal{Q}_{i}$ is the
		previously introduced $q$-th order accurate quadrature rule over $\Omega_{i}$
		with nodes $x_{i,\alpha}$ and weights $\omega_{\alpha}$.
		We again stress that the equilibrium density $\rho_{i}^{\eq}$ and the
		gravitational acceleration $g_i^{\rm int}$ are polynomials and, consequently,
		almost everything can be evaluated analytically in a straightforward manner.
		Only the EoS conversion to internal energy density given density and pressure
		$\varepsilon = \varepsilon(\rho,p)$ is in general not explicitly available.
		Therefore, solving \cref{eq:appendix_wb_rec_0010} for $p_{0,i}$ requires some
		iterative procedure such as Newton's method
		\begin{equation}
		\label{eq:appendix_wb_rec_0030}
		p_{0,i}^{(k+1)} = p_{0,i}^{(k)}
		- \frac{f(p_{0,i}^{(k)})}{f'(p_{0,i}^{(k)})}
		, \quad (k=0,1,\dots)
		\end{equation}
		where the superscript in parenthesis labels the iteration number and the
		derivative of \cref{eq:appendix_wb_rec_0020} is given by
		\begin{equation}
		\label{eq:appendix_wb_rec_0040}
		f'(p)
		= - \frac{1}{\Delta x}
		\sum_{\alpha=1}^{N_{q}}
		\omega_{\alpha}
		~
		\frac{\del \varepsilon}{\del p}
		\left(  \rho_{i}^{\eq}(x_{i,\alpha})
		, ~ p
		+ \int_{x_{i}}^{x_{i,\alpha}}  s_{i}(\xi) ~ \d \xi
		\right)
		.
		\end{equation}
		The iteration is started with the pressure computed from the cell-averaged
		conserved variables as initial guess
		\begin{equation}
		\label{eq:appendix_wb_rec_0050}
		p_{0,i}^{(0)} = p(\hat\rho_{i},\hat\varepsilon_{i})
		.
		\end{equation}
		The iteration is stopped and the cell-centered pressure $p_{0,i}=p_{0,i}^{(k)}$ returned if the condition
		$$
		\left|\frac{f\left(p_{0,i}^{(k)}\right)}{f^\prime\left(p_{0,i}^{(k)}\right)}\right|<\tau
		$$
		is met, where we chose $\tau=10^{-13}$ in the numerical experiments conducted in this article.
		
		As is well-known, the global convergence properties of Newton's method are poor. However, it is straightforward to build a robust solver by combining it with, for example, the bisection method~(see \cite{DennisSchnabel1996,Press2007} and references therein for details).
		Such a modification was not necessary for the presented numerical examples
		using the ideal gas with radiation pressure EoS.
		
		
		
		\paragraph{Simplified approach}
		However, for many applications it might be sufficient to use a simplified approach to determine the value of $p_{0,i}$. Choose
		$$\varepsilon_{0,i}:=E^\text{rec}_i(x_i) - \half \frac{\left((\rho u)^\text{rec}_i(x_i)\right)^2}{\rho^\text{rec}_i(x_i)},$$
		which is the cell-centered internal energy computed from the CWENO reconstruction polynomials. Then apply the EoS to compute
		$$p_{0,i}:=p_\text{EoS}\left(\rho^\text{rec}_i(x_i),\varepsilon_{0,i}\right).$$ The resulting method will be referred to as 
		\dwbeosf.
		
		\subsubsection{Well-balanced property of the \tdwb}
		We summarize the well-balancing property of the \dwb method in the following theorem.
		\begin{theorem}
			\label{thm:nm_1d}
			Consider the scheme \eqref{eq:nm_1d_0040} with a consistent, Lipschitz
			continuous, and contact property fulfilling (\cref{def:contact_property})
			numerical flux $\mathcal{F}$, an $m$-th order accurate spatial reconstruction
			procedure $\mathcal{R}$, a $q$-th order accurate quadrature rule $\mathcal{Q}$, 
			the hydrostatic reconstruction $\mathcal{W}$ (\cref{eq:nm_1d_wb_0120}) and the
			standard gravitational source term discretization (\cref{eq:nm_1d_0100}). 
			
			This scheme has the following properties:
			\begin{enumerate}
				\item[(i) ] The scheme is consistent with \cref{eq:nm_1d_0010} and it is 
				$\min(q,m)$-th order accurate in space (for smooth solutions).
				\item[(ii)] The scheme is well-balanced in the sense that it exactly preserves
				a discrete hydrostatic equilibrium approximating an arbitrary
				non-periodic hydrostatic equilibrium
				$\w^{\eq} = [\rho^{\eq}, 0, p^{\eq}]^{T}$
				to $\min(q,m)$-th order accuracy (for smooth equilibrium
				$\w^{\eq}$).
			\end{enumerate}
		\end{theorem}
		\begin{proof}
			(i) 
			\refb{The overall accuracy of the scheme is determined by the accuracy of
			reconstruction and source term integration. It is straightforward to show the source term discretization is $\min (q,m)$-th order accurate. The order of the accuracy of the hydrostatic reconstruction in energy, however, requires a discussion.
			
			Assume a smooth solution $\tilde q=\left( \tilde{\rho},\tilde{\rho u},\tilde E \right)^T$ at a fixed time $t$ (which we omit in the following), with the pressure value $\tilde p_{0,i}:=\tilde E(x_i)-\tilde {(\rho u)}(x_{i})^2/\left(2\tilde{\rho}(x_i)\right)$ at the point $x_i$. The functions $\tilde p_i^\eq$ and $\tilde \varepsilon_i^\eq$ defined by $\tilde p_i^\eq(x):=\tilde p_{0,i}+\int_{x_i}^x \tilde{\rho}(\xi)g(\xi)\,d\xi$ and $\tilde \varepsilon_i^\eq(x):=\varepsilon_\text{EoS}\left(\tilde{\rho}(x),\tilde p_i^\eq (x)\right)$ respectively are then smooth functions provided that the EoS is smooth.
			
			In the following step we show that the piecewise smooth equilibrium energy profile $\varepsilon_i^\eq$ that is used in our method approximates the smooth profile $\tilde\varepsilon_i^\eq$ to $\min(q,m)$-th order: By construction, the cell-averaged internal energy estimate $\hat{\varepsilon}_i$ (\cref{eq:nm_1d_wb_0070}) is $m$ order accurate and $p_{0,i}$ obtained as described in \cref{eq:nm_1d_wb_0080} or the iterative procedure in \cref{sec:discrete_wb_eos} approximates $\tilde p_{0,i}$ to $\mu:=\min(q,m)$-th order. Consequently, we have
			\begin{align*}
				 p_i^\eq (x)
				 &= p_{0,i}+\int_{x_i}^x {\rho}^\rec(\xi)g^\text{int}(\xi) \,d\xi
				\\
				&= \tilde p_{0,i} + \order\left((x-x_i)^\mu\right)+\int_{x_i}^x \left(\tilde{\rho}(\xi)+ \order\left((x-x_i)^m\right)\right)\left(g(\xi) + \order\left((x-x_i)^m\right)\right)\,d\xi
				\\
				&= \tilde p_{0,i}+\int_{x_i}^x \tilde{\rho}(\xi)g(\xi)\,d\xi
				 + \order\left((x-x_i)^\mu\right)
				 = \tilde p_i^\eq(x) + \order\left((x-x_i)^\mu\right)
			\end{align*}
			and thus
			\begin{align*}
				\varepsilon_i^\eq (x)
				&= \varepsilon_\text{EoS}\left(\rho^\rec(x),p_i^\eq(x)\right)
				\\
				&= \varepsilon_\text{EoS}\left(\tilde\rho(x) + \order\left((x-x_i)^\mu\right),\tilde p_i^\eq(x)+\order\left((x-x_i)^\mu\right)\right)
				\\
				&= \varepsilon_\text{EoS}\left(\tilde\rho(x),\tilde p_i^\eq(x)\right)+\order\left((x-x_i)^\mu\right)
			\end{align*}
			for sufficiently smooth EoS.
			
			With this we can finally show the $\mu$-th order accuracy of the hydrostatic reconstruction in the energy:
			\begin{align}
			E^{\rec}_i(x) &= \varepsilon^{eq}_i(x) + \mathcal R\left(x;\left\{\hat E_k-\frac1{\Delta x}\mathcal Q_k\left( \varepsilon^{\eq}_i\right)\right\}_{j\in\mathcal S_i}\right) \nonumber\\
			&= \tilde  \varepsilon^{eq}_i(x)+\order\left(\Delta x^\mu\right)+ \mathcal R\left(x;\left\{\hat E_k-\frac1{\Delta x}\int_{\Omega_k} \varepsilon^{\eq}_i(\xi)\, d\xi + \order\left(\Delta x^\mu\right)\right\}_{j\in\mathcal S_i}\right) \nonumber\\
			\label{eq:proof_1}
			&= \tilde  \varepsilon^{eq}_i(x)+ \mathcal R\left(x;\left\{\hat E_k-\frac1{\Delta x}\int_{\Omega_k} \varepsilon^{\eq}_i(\xi)\, d\xi \right\}_{j\in\mathcal S_i}\right)+\order\left(\Delta x^\mu\right) \\
			\label{eq:proof_2}
			&= \tilde \varepsilon^{eq}_i(x) + \left(\tilde E(x)-\tilde \varepsilon_i^{eq}(x)\right) + \mathcal O(\Delta x^\mu)
			 = \tilde E(x) + \mathcal O(\Delta x^\mu)
			\end{align}
			for $x\in\Omega_i$. In \cref{eq:proof_1} the Lipshitz continuity of the reconstruction $\mathcal R$ (which is given for  consistent reconstruction methods on smooth flows) is applied and the step to \cref{eq:proof_2} holds because $\varepsilon_i^\eq$ is a smooth function in the whole reconstruction stencil $\mathcal S_i$ and $\mathcal R$ is $m$-th order accurate. This concludes the proof that the hydrostatic reconstruction in the energy is $\mu=\min(q,m)$-th order accurate.
				
			The consistency of the scheme follows directly from the order of accuracy.
			}
			
			(ii) The proof of this item consists of two parts.
			First, we construct a discrete equilibrium, consistent with the local
			equilibrium reconstruction procedure $\mathcal{W}$, and show that it
			approximates an arbitrary hydrostatic equilibrium with high-order accuracy.
			Second, we show that the just constructed discrete hydrostatic equilibrium is
			exactly preserved by the scheme.
			
			We begin by part one.
			Let an arbitrary (but smooth enough) hydrostatic equilibrium be given
			\begin{equation*}
			{\w}^{\eq}(x) = [\rho^{\eq}(x), 0,   p^{\eq}(x)]^{T}
			\end{equation*}
			with gravitational acceleration $g(x)$.
			The corresponding equilibrium conserved variables are then
			\begin{equation*}
			 \q^{\eq}(x) = [ \rho^{\eq}(x), 0, \varepsilon_\text{EoS}( \rho^{\eq}(x), p^{\eq}(x))]^{T}
			.
			\end{equation*}
			We stress that these are exact profiles\footnote{\refb{Note that $\rho^{\eq}(x)$ is the exact hydrostatic profile, which is generally unknown, whereas $\rho^{\eq}_i(x)$ is a locally reconstructed profile.}}.
			Let the density cell averages in every cell be given by the $q$-th order
			accurate quadrature rule $\mathcal{Q}_{i}$
			\begin{equation*}
			\hat{{\rho_{i}}} = \frac{1}{\Delta x} \mathcal{Q}_{i}( \rho^{\eq})
			.
			\end{equation*}
			By applying the $m$-th order accurate standard reconstruction procedure
			$\mathcal{R}$ to the density cell averages $\hat{{\rho}}_{i}$,
			\begin{equation*}
			\rho_{i}^{\rec}(x) = \mathcal{R}\left(x; \{\hat{{\rho}}_{k}\}_{k\in\mathcal S_{i}}\right)
			,
			\end{equation*}
			we obtain a $\min(q,m)$-th order accurate approximation of $ \rho^{\eq}(x)$
			within every cell $\Omega_{i}$.
			Because the local equilibrium density profile $\rho_{i}^{\eq}(x)$
			coincides with $\rho_{i}^{\rec}(x)$ in the scheme, it approximates
			$ \rho^{\eq}(x)$ with the same accuracy.
			
			Let us now focus on a particular cell $\Omega_{i}$ and anchor the local
			equilibrium pressure profile at its center by setting
			$p_{0,i} =  p^{\eq}(x_{i})$ in \cref{eq:nm_1d_wb_0050}:
			\begin{equation*}
			p_{i}^{\eq}(x) 
			=  p^{\eq}(x_i) + \int_{x_{i}}^{x} s_{i}(\xi) \d \xi
			.
			\end{equation*}
			We emphasize that exact integration is used in the definition of the local
			equilibrium profile.
			This is straightforward because $\rho_{i}^{\eq}(x)$ and $g_i^{\rm int}$ are simply polynomials.
			Then it is clear that the above $p_{i}^{\eq}(x)$ is a $\min(q,m)$-th order
			approximation of $ p^{\eq}(x)$ within this particular cell.
			With the local equilibrium density and pressure profile available, we readily
			obtain the internal energy density through the EoS with
			\cref{eq:nm_1d_wb_0030}.
			Applying the quadrature rule $\mathcal{Q}_{i}$ as in \cref{eq:nm_1d_wb_0060},
			we obtain the cell-averaged internal energy density within cell $\Omega_{i}$.
			Note that the so obtained cell-averaged conserved variables within the $i$-th
			cell are $\min(q,m)$-th order accurate approximation of the exact cell-averaged
			equilibrium conserved variables, i.e.
			\begin{equation*}
			\hat{{\q}}_{i}^{\eq}
			= \frac{1}{\Delta x} \int_{\Omega_{i}}  \q^{\eq}(x) \d x
			= [\hat{ \rho}_{i}, 0, \hat{{\varepsilon}}_{i}]^{T}
			+ \mathcal{O}(\Delta x^{\min(q,m)})
			= \hat{{\Q}}_{i}^{\eq} + \mathcal{O}(\Delta x^{\min(q,m)})
			.
			\end{equation*}
			These are the discrete equilibrium cell-averaged conserved variables within
			this particular $i$-th cell obtained from $  \q^{\eq}(x)$.
			
			Next, we construct the discrete equilibrium cell-averaged conserved variables
			in the remaining cells, i.e. all other cells than $\Omega_{i}$.
			It would seem that one could simply repeat the above procedure for the
			internal energy density in every cell.
			However, this will not result in a consistent discrete equilibrium.
			Indeed, we need to make sure that the discrete equilibrium in all cells is
			consistent with the particular cell $\Omega_{i}$, where we anchored
			the equilibrium, and also among each other.
			To achieve this, we extrapolate the local equilibrium profile from the
			particular cell $\Omega_{i}$ to its immediate neighbors by enforcing pressure
			equality at the touching cell interfaces.
			Indeed, if the pressure is not equal at a cell interface, then there is no
			equilibrium and a net force arises.
			Operationally, this consistent equilibrium extrapolation from cell $\Omega_{i}$
			to any cell $\Omega_{j}$ can be written as
			\begin{equation}
			\label{eq:wbcondition}
			p_{0,j} 
			= p_{0,i}
			+ \int_{x_{i}}^{x_{j}} s_{h}(\xi) ~ \d \xi
			,
			\end{equation}
			where
			\begin{equation*}
			s_{h}(x) = \sum_{k} \rho^\text{eq}_k(x)g^{\rm int}_k(x) ~ \mathbbm{1}_{\Omega_{k}}(x) \qquad\textrm{and}\qquad
			\mathbbm{1}_{\Omega_{k}}(x) = \begin{cases}
			1 \text{ if } x \in     \Omega_{k} \\
			0 \text{ if } x  \notin \Omega_{k}
			\end{cases}
			\end{equation*}
			is the characteristic function of the $k$-th cell.
			This allows the definition of the discrete equilibrium cell averages in all
			cells $\hat{ \Q}_{j}^{\eq}$ consistent with the equilibrium in cell $\Omega_{i}$
			where it was anchored.
			For all the cells we have now
			\begin{equation*}
			\hat{{\q}}_{j}^{\eq}
			= \frac{1}{\Delta x} \int_{\Omega_{j}}  \q^{\eq}(x) \d x
			= \hat{\Q}_{j}^{\eq} + \mathcal{O}(\Delta x^{\min(q,m)})
			\end{equation*}
			because it is clear that
			\begin{equation*}
			 p^{\eq}(x_{j}) = p_{j,0} + \mathcal{O}(\Delta x^{\min(q,m)})
			.
			\end{equation*}
			
			Reciprocally, it is easy to obtain the same local equilibrium profiles
			from the discrete equilibrium cell-averaged conserved variables
			$\hat\Q_{i}^{\eq}$.
			To guarantee this, we need to make sure that, given such equilibrium
			cell-averaged conserved variables, the same equilibrium pressure at cell center
			$p_{0,i}$ in \cref{eq:nm_1d_wb_0050} is obtained when solving
			\cref{eq:nm_1d_wb_0060}.
			For the ideal gas EoS this is clear, because there is only solution
			\cref{eq:nm_1d_wb_0080}.
			For general EoS, this is shown in \ref{sec:appendix_wb_rec} with the
			requirement that the equilibrium within the cell is away from any phase
			transition.
			
			We now conclude this item with part two.
			Given discrete cell-averaged equilibrium conserved variables $\hat\Q_{i}^{\eq}$
			as constructed in part one, the well-balanced reconstruction procedure
			$\mathcal{W}$ guarantees that matching local equilibrium profiles
			$\Q_{i}^{\eq}(x)$ are found within each cell and that, consequently, the
			equilibrium perturbations $\delta \Q_{i}^{\rec}(x)$ vanish in all cells.
			Therefore, we have pressure equilibrium at cell interfaces
			\begin{equation*}
			p_{i  }^{\eq}(x_{i+\half})
			=   p_{0,i  }
			+ \int_{x_{i  }}^{x_{i+\half}} s_{i  }(\xi) \d \xi
			=   p_{0,i+1}
			+ \int_{x_{i+1}}^{x_{i+\half}} s_{i+1}(\xi) \d \xi
			= p_{i+1}^{\eq}(x_{i+\half})
			.
			\end{equation*}
			A contact property fulfilling numerical flux $\mathcal{F}$ then results in
			\begin{equation*}
			\mathcal{F}\left(
			\Q(\rho_{i  }^{\eq}(x_{i+\half}), 0, p_{i  }^{\eq}(x_{i+\half}))
			, \Q(\rho_{i+1}^{\eq}(x_{i+\half}), 0, p_{i+1}^{\eq}(x_{i+\half}))
			\right)
			=
			\left[ 0,p_{i+\half} , 0 \right]^{T}
			\end{equation*}
			with
			\begin{equation*}
			p_{i+\half} = p_{i  }^{\eq}(x_{i+\half}) = p_{i+1}^{\eq}(x_{i+\half})
			.
			\end{equation*}
			Plugging this into the momentum component of the scheme's flux difference
			\cref{eq:nm_1d_0040}
			\begin{equation*}
			\begin{aligned}
			\frac{1}{\Delta x}
			\left[ F_{i+\half}^{[\rho u]} - F_{i-\half}^{[\rho u]} \right]
			& = \frac{1}{\Delta x} \left[ p_{i+\half} - p_{i-\half} \right] \\
			& 
			= \frac{1}{\Delta x}
			\left[
			\left(
			p_{0,i}
			+ \int_{x_{i}}^{x_{i+\half}} s_{i}(\xi) \d \xi
			\right)
			- \left(
			p_{0,i}
			+ \int_{x_{i}}^{x_{i-\half}} s_{i}(\xi) \d \xi
			\right)
			\right]
			\\
			& 
			= \frac{1}{\Delta x}
			\int_{x_{i-\half}}^{x_{i+\half}} s_{i  }(\xi) \d \xi
			\\
			& = \hat S_{i}^{[\rho u]}
			.
			\end{aligned}
			\end{equation*}
			The density and energy component flux difference vanish as does the energy
			component source term.
			To conclude, we obtain
			\begin{equation*}
			\dt\hat\Q^{\eq}_i(t)
			= \mathcal{L}\left(\hat\Q^{\eq}_i\right)
			= - \frac{1}{\Delta x} \left[ \F_{i+\half} -\F_{i-\half}\right]
			+ \hat\S_{i}
			= 0
			\end{equation*}
			and the scheme is well-balanced as claimed.

		\end{proof}
		\begin{remark}
			Periodic hydrostatic solutions have been excluded in \cref{thm:nm_1d} since the construction of the discrete approximation which is used in the proof can fail for periodic boundary conditions. Periodic hydrostatic states are anyway academic problems, since they can not appear in real physical situations. However, the \dwb method can still be beneficially applied to periodic hydrostatic states as will be demonstrated in \cref{sec:1d_isothermal,sec:1d_isothermal_pert,sec:non-barotropic}.
		\end{remark}
		\begin{remark}
			For the equilibrium density reconstruction we chose to use the same
			reconstruction procedure $\mathcal{R}$ as used in the equilibrium perturbation
			reconstruction.
			In principle, a different reconstruction procedure $\mathcal{R}^{eq}$ for the
			equilibrium density can be used.
			However, we do not further explore this possibility in this work.
		\end{remark}
		\begin{remark}
			The schemes developed here reduce to the second-order accurate scheme presented
			in \cite{Kaeppeli2015,Kaeppeli2016} when setting the quadrature rule
			$\mathcal{Q}$ to the midpoint rule, the equilibrium density reconstruction
			$\rho_{i}^{eq}(x)$ to piecewise constant and the equilibrium perturbation
			reconstruction procedure to piecewise linear.
		\end{remark}
		
		\refb{ 
			\subsubsection{Well-balanced boundary conditions}
			\label{sec:dwb_bcs}
			
			In the following we discuss the different kinds of boundary conditions and how to realize them in order to comply with the well-balancing property as stated in \cref{thm:nm_1d}. All of the boundary conditions we discuss are based on a sufficient number of \emph{ghost cells}, which have to be added on either side of the domain and which are set to certain values before reconstruction, depending on the chosen boundary condition.
			
			\paragraph{Dirichlet boundary conditions} If the initial data of a simulation satisfy the relation \cref{eq:wbcondition} in the whole domain and in the ghost cells, Dirichlet boundary conditions can be realized by simply never updating the ghost cell values.
			
			\paragraph{Hydrostatic extrapolation}
			First, the states in the ghost cells $\Omega_j$ for $j\in\left\{1-N_{gc},...,\Omega_{0}\right\}$ (for brevity we only describe the left boundary) are set to
			$$\hat\Q_{j}=\frac1{\Delta x}\int_{\Omega_j} \Q^\text{rec}_{1+\frac{m-1}2}(x)\,dx.$$
			We compute $p_{0,1}$ according to \cref{eq:nm_1d_wb_0060} and the hydrostatic pressure
			$$p_1^\text{eq}(x)=p_{0,1}+\int_{x_1}^xs_1(\xi)\,d\xi,$$ where $s_1$ is defined as in \cref{def:source_term_dwb}. The total energy in each cell is then corrected to $$\hat E_j=\frac1{\Delta x}\int_{\Omega_j}\varepsilon\left(\rho_1^\text{rec}(x),p_1^\text{eq}(x)\right) + \frac12 \frac{\left((\rho u)_1^\text{rec}(x)\right)^2}{\rho_1^\text{rec}(x)}\, dx$$
			to achieve a well-balanced treatment of the boundaries.

			\paragraph{Solid wall boundaries}
			Before reconstructing in every cell, fill the ghost cells with the data obtained from the hydrostatic reconstruction discussed above. After reconstruction, use the following boundary fluxes:
			\begin{align*}
			\F_{\half}&= \mathcal F\left(\begin{pmatrix}
			\rho^L_1\\(\rho u)^L_1\\E^L_1
			\end{pmatrix},\begin{pmatrix}
			\rho^L_1\\-(\rho u)^L_1\\E^L_1
			\end{pmatrix}\right),
			&
			\F_{N+\half}&= \mathcal F\left(\begin{pmatrix}
			\rho^R_N\\(\rho u)^R_N\\E^R_N
			\end{pmatrix},\begin{pmatrix}
			\rho^R_N\\-(\rho u)^R_N\\E^R_N
			\end{pmatrix}\right),
			\end{align*}
			where $\left(\rho_1^L,(\rho u)^L_1,E^L_1\right)^T=\hat\Q^\text{rec}_1\left(x_{\half}\right)$
			and
			$\left(\rho^R_N,(\rho u)^R_N,E^R_N\right)^T=\hat\Q^\text{rec}_N\left(x_{N+\half}\right)$
			are obtained via the well-balanced reconstruction procedure $\mathcal W$.
		}

		\subsubsection{Stencil of the \tdwb}
		\label{sec:dwb_stencil}
		
		In the whole article, we have assumed that the order of reconstruction $m$ is  odd, due to the use of CWENO schemes.  To update the cell-average values $\hat{\Q}_i$, a standard $m$-th order method requires $\hat{\Q}_{i-\frac{m+1}2},\dots,\hat{\Q}_{i+\frac{m+1}2}$. This includes $\frac{m-1}2$ cells in each direction for the reconstruction and one for the flux computations from the reconstructed values in the $i-1,i,$ and $i+1$ cell.
		
		The \dwb method increases the stencil in the following way. The transformation to local hydrostatic variables requires the values of $s_h$ in each cell in the reconstruction stencil. This adds $\frac{m-1}2$ cells in each direction to the stencil. In total, to update the cell-average values $\hat{\Q}_i$, the methods require the values $\hat{\Q}_{i-m},\dots,\hat{\Q}_{i+m}$. The stencil is visualized in \cref{fig:stencil_dwb}. Depending on the application (especially in parallel computing using a domain decomposition), this increased stencil can lead to a considerable increase in computation time and memory requirements. As a possible solution to this problem, we propose a modified method in the next section. 
		\begin{figure}
			\centering
			\begin{tikzpicture}
			\pgfmathsetmacro{\ticklen}{0.1};
			\pgfmathsetmacro{\xval}{-5};
			\interface{\xval,0}{i-25}{};\draw (\xval ,-\ticklen)--(\xval ,\ticklen);
			\pgfmathsetmacro{\xval}{-3};
			\interface{\xval,0}{i-15}{};\draw (\xval ,-\ticklen)--(\xval ,\ticklen);
			\pgfmathsetmacro{\xval}{-1};
			\interface{\xval,0}{i-05}{};\draw (\xval ,-\ticklen)--(\xval ,\ticklen);
			\pgfmathsetmacro{\xval}{1};
			\interface{\xval,0}{i+05}{};\draw (\xval ,-\ticklen)--(\xval ,\ticklen);
			\pgfmathsetmacro{\xval}{3};
			\interface{\xval,0}{i+15}{};\draw (\xval ,-\ticklen)--(\xval ,\ticklen);
			\pgfmathsetmacro{\xval}{5};
			\interface{\xval,0}{i+25}{};\draw (\xval ,-\ticklen)--(\xval ,\ticklen);

			\cellcenter{-4,0}{i-2}{$i-2$};
			\cellcenter{-2,0}{i-1}{$i-1$};
			\cellcenter{ 0,0}{i}  {$i$};
			\cellcenter{ 2,0}{i+1}{$i+1$};
			\cellcenter{ 4,0}{i+2}{$i+2$};
			
			\draw 
			(-5,0) -- (5,0);
			
			\draw [decorate,decoration={brace,amplitude=10pt,mirror},yshift=-.5cm](-5,0.) -- (1,0.) node[black,midway,yshift=-0.5cm] { \sourcedef{i-1}};
			
			\draw [decorate,decoration={brace,amplitude=10pt,mirror},yshift=-1.3cm](-1.,0.) -- (5,0.) node[black,midway,yshift=-0.5cm] {\sourcedef{i+1}};
			
			\draw [decorate,decoration={brace,amplitude=10pt,mirror},yshift=-2cm](-3,0.) -- (3,0.) node[black,midway,yshift=-0.5cm] { \sourcedef{i}};

			\pgfmathsetmacro{\yval}{-4.1};
			
			\draw (-3,\yval+0.8) edge[parabola through={(-2,\yval+0.4)},
			thick,fill opacity=.21] (-1,\yval+0.3);
			\cellcenter{-2,\yval+0.4}{parabola1}{};	
			\draw (-1,\yval+.25) edge[parabola through={(0,\yval+0.2)},
			thick,fill opacity=.21] (1,\yval+.35);
			\cellcenter{0,\yval+0.2}{parabola2}{$\source{i}$};
			\draw (1,\yval+0.25) edge[parabola through={(2,\yval+0.5)},
			thick,fill opacity=.21] (3,\yval+01.1);
			\cellcenter{2,\yval+0.5}{parabola3}{};
			
			\pgfmathsetmacro{\ticklen}{0.1};
			\pgfmathsetmacro{\yval}{-4.6};
			\pgfmathsetmacro{\xval}{-3};\draw (\xval ,\yval-\ticklen)--(\xval ,\yval+\ticklen);
			\pgfmathsetmacro{\xval}{-1};\draw (\xval ,\yval-\ticklen)--(\xval ,\yval+\ticklen);
			\pgfmathsetmacro{\xval}{ 1};\draw (\xval ,\yval-\ticklen)--(\xval ,\yval+\ticklen);
			\pgfmathsetmacro{\xval}{ 3};\draw (\xval ,\yval-\ticklen)--(\xval ,\yval+\ticklen);
			\draw 
			(-3,\yval) -- (3,\yval);
			
			\cellcenter{-2,\yval}{j-1}{$i-1$};
			\cellcenter{ 0,\yval}{j}  {$i$};
			\cellcenter{ 2,\yval}{j+1}{$i+1$};
			
			\draw [decorate,decoration={brace,amplitude=10pt,mirror},yshift=-.5cm](-3.,\yval) -- (3,\yval) node[black,midway,yshift=-0.6cm] {\footnotesize $p_i^{\rm eq}$, ${\mathbf{Q}}^{\rm rec}_i$};
			
			\end{tikzpicture}
			\caption{Stencil of the third order accurate \tdwb method. The local hydrostatic reconstruction which yields $\Q_i^\text{rec}$ requires the source term approximations $\s_{i-1}$, $s_{i}$,$s_{i+1}$ in the $i-1$, $i$, and $i+1$ cell respectively (shown at the bottom of the figure). Each of these source term approximation has a stencil involving one neighboring cell per dimension. The total stencil to determine $\Q_i^\text{rec}$ thus involves five cells.}
			\label{fig:stencil_dwb}
		\end{figure}
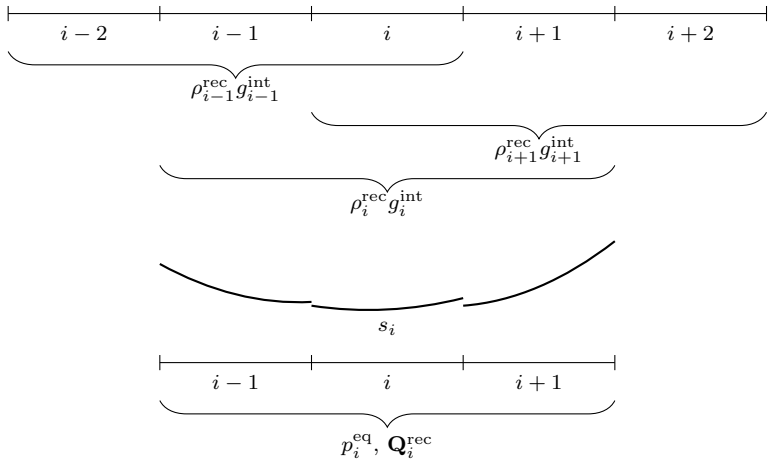
		
		\subsection{The \tla}
		\label{sec:compact_method}
		
		The reason for the increased stencil in the previous methods is that the source term has to be discretized in each cell of the CWENO stencil.  To avoid this, we will now do the following. To compute the hydrostatic pressure with respect to the $i$-th cell, we only use the source term discretization from the $i$-th cell. This definition is extended to the whole domain in a trivial way without using additional information. Consequently, there is no unique source term discretization; instead it depends on the cell in which we aim to reconstruct.
		
		To achieve this, we only have to modify \cref{def:source_term_dwb} to
		\begin{equation}
		\label{def:source_term_la}
		s_{i}(x) = \rho_{i}^{\eq}(x)g_i^{\rm int}(x)\quad\text{for}
		\quad
		x\in\bigcup_{k\in\mathcal S_i}\Omega_k
		\end{equation}
		for performing the hydrostatic reconstruction in the $i$-th cell.
		Thus, we extrapolate the source term polynomial from the $i$-th cell to the neighboring cells. This only effects the reconstruction of the energy deviations. The rest of the method remains unmodified. In the rest of the article, we refer to this modified method as local approximation (LA) method. Obviously, this method can also be applied for non-ideal EoS using the procedures described in \cref{sec:discrete_wb_eos}. The resulting methods will be referred to as \laeos when Newton's method is applied to determine $p_{0,i}$ and \laeosf if the simplified method is applied.
		
		\subsubsection{Stencil of the \tla}
		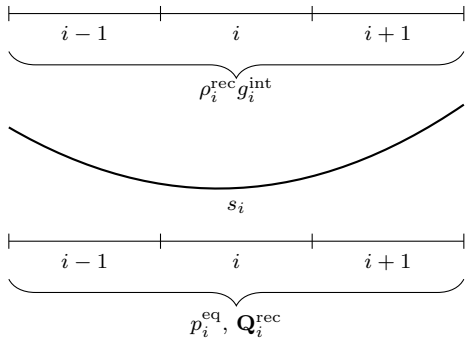
\begin{figure}
			\centering
			\begin{tikzpicture}
			\pgfmathsetmacro{\ticklen}{0.1};
			\pgfmathsetmacro{\xval}{-3};
			\interface{\xval,0}{i-15}{};\draw (\xval ,-\ticklen)--(\xval ,\ticklen);
			\pgfmathsetmacro{\xval}{-1};
			\interface{\xval,0}{i-05}{};\draw (\xval ,-\ticklen)--(\xval ,\ticklen);
			\pgfmathsetmacro{\xval}{1};
			\interface{\xval,0}{i+05}{};\draw (\xval ,-\ticklen)--(\xval ,\ticklen);
			\pgfmathsetmacro{\xval}{3};
			\interface{\xval,0}{i+15}{};\draw (\xval ,-\ticklen)--(\xval ,\ticklen);

			\cellcenter{-2,0}{i-1}{$i-1$};
			\cellcenter{ 0,0}{i}  {$i$};
			\cellcenter{ 2,0}{i+1}{$i+1$};
			
			\draw 
			(-3,0) -- (3,0);
			
			\draw [decorate,decoration={brace,amplitude=10pt,mirror},yshift=-.5cm](-3.,0.) -- (3,0.) node[black,midway,yshift=-0.5cm] {\sourcedef{i}};

			\pgfmathsetmacro{\yval}{-2.5};
			
			\draw (-3,\yval+1.) edge[parabola through={(0,\yval+0.2)},
			thick,fill opacity=.21] (3,\yval+1.3);
			\cellcenter{0,\yval+0.2}{parabola1}{$\source{i}$};	
			
			\pgfmathsetmacro{\ticklen}{0.1};
			\pgfmathsetmacro{\yval}{-3.};
			\pgfmathsetmacro{\xval}{-3};\draw (\xval ,\yval-\ticklen)--(\xval ,\yval+\ticklen);
			\pgfmathsetmacro{\xval}{-1};\draw (\xval ,\yval-\ticklen)--(\xval ,\yval+\ticklen);
			\pgfmathsetmacro{\xval}{ 1};\draw (\xval ,\yval-\ticklen)--(\xval ,\yval+\ticklen);
			\pgfmathsetmacro{\xval}{ 3};\draw (\xval ,\yval-\ticklen)--(\xval ,\yval+\ticklen);
			\draw 
			(-3,\yval) -- (3,\yval);
			
			\cellcenter{-2,\yval}{j-1}{$i-1$};
			\cellcenter{ 0,\yval}{j}  {$i$};
			\cellcenter{ 2,\yval}{j+1}{$i+1$};
			
			\draw [decorate,decoration={brace,amplitude=10pt,mirror},yshift=-.5cm](-3.,\yval) -- (3,\yval) node[black,midway,yshift=-0.6cm] {\footnotesize $p_i^{\rm eq}$, ${\mathbf{Q}}^{\rm rec}_i$};
			
			\end{tikzpicture}
			\caption{Stencil of the third order accurate \tla method. Different from the \tdwb method (see stencil in \cref{fig:stencil_dwb}), the local hydrostatic reconstruction only requires the source term approximation computed in the $i$-th cell to compute $\Q_i^\text{rec}$. The source term approximation $s_i$ is for this purpose extrapolated to the neighboring cells. Thus the total stencil of the reconstruction only includes three cells, equivalent to a non-well-balanced standard method.}
			\label{fig:stencil_la}
		\end{figure}
		In this modified method, the reconstruction routine only requires the local hydrostatic pressure polynomial from the $i$-th cell. The stencil of the method is now the same as the stencil of the standard method of the same formal order of accuracy. It is visualized in \cref{fig:stencil_la}.
		
		\subsubsection{Well-balanced property of the \tla}
		For the \la method defined in \cref{sec:compact_method}, there is no globally defined hydrostatic pressure function. Consequently, in general there is no well-defined cell-to-cell relation like \cref{eq:wbcondition}, which is balanced to machine precision. The relation only holds if the hydrostatic pressure polynomials in different cells can be described as one global polynomial. Whether the \la methods actually succeeds in significantly reducing the discretization error at hydrostatic solutions has to be tested in the numerical experiments.
		\subsection{Summary of the scheme}
		The well-balancing techniques we propose are within the framework of high order Runge--Kutta finite volume methods. To make it more evident, we outline all of the steps necessary to obtain the interface states and source term discretization from the cell-averaged states $\hat\Q_i$.
		\begin{enumerate}
			\item 
			Reconstruct density and momentum and interpolate the gravitational acceleration in each cell to obtain $\rho^\text{rec}_i$, $(\rho u)^\text{rec}_i$, and $g^\text{int}_i$.
			\item 
			Compute the source term representation $s_i$ as defined in
			\begin{itemize}
				\item[(a)] \cref{def:source_term_dwb} for the \dwb method.
				\item[(b)] \cref{def:source_term_la} for the \la method.
			\end{itemize}
			The source term discretization $\hat\S_i$ is obtained from \cref{eq:nm_1d_0100}.
			\item 
			Define the hydrostatic pressure polynomial $p_i^\text{eq}$ for the $i$-th cell according to \cref{eq:nm_1d_wb_0050}.
			For this, the value of $p_{0,i}$ is obtained from
			\begin{itemize}
				\item[(i)] \cref{eq:nm_1d_wb_0080} if an ideal gas EoS is used to close the Euler system.
				\item[(ii)] one of the methods presented in \cref{sec:discrete_wb_eos} if any other EoS is used to close the Euler system.
			\end{itemize}
			\item 
			Define the cell-averaged high order accurate representation of the equilibrium conserved variables $\hat\Q_i^\text{eq}$ for each cell $\Omega_i$ as in \cref{eq:nm_1d_wb_0010} and apply the hydrostatic reconstruction routine given in \cref{eq:nm_1d_wb_0110,eq:nm_1d_wb_0120} to obtain the interface states.
		\end{enumerate}
		We like to repeat and emphasise at this point, that all of the steps above are \emph{local} such that the methods are suitable to be implemented in parallelized codes. It is not necessary to complete each step on the whole grid before commencing to the next one.

		\section{Numerical experiments in one spatial dimension}
		\label{sec:ne}
		
		In all numerical experiments in this section we use the standard Roe flux~\cite{Roe1981}. For the third order methods we use the third order accurate CWENO3 reconstruction proposed in \cite{Kolb2014} and the third-order accurate, four
		stage explicit Runge--Kutta method from \cite{Kraaijevanger1991}. For the fifth order accurate methods we use the fifth order accurate CWENO5 reconstruction proposed in \cite{Capdeville2008} and the fifth order accurate standard Runge--Kutta method from \cite{Rabiei2012}. 
		
		The results are often compared to the results obtained with a {\em standard scheme}, \ie a non-well-balanced scheme. For this we use exactly the same methods as for the \la and \dwb scheme with the only difference that instead of the hydrostatic reconstruction procedure $\mathcal W$ (\cref{eq:nm_1d_wb_0120}) the standard reconstruction procedure $\mathcal R$ on the conserved quantities is applied.

		\subsection{Isothermal hydrostatic solution}
		\label{sec:1d_isothermal}
		We consider an isothermal hydrostatic solution of the 1-d compressible Euler equations with gravitational source term and the ideal gas equation of state (we choose $\gamma=1.4$) given by
		\begin{equation}
		\tilde\rho(x) = \tilde p(x) = \exp(-\phi(x)),
		\qquad \tilde u \equiv 0.
		\label{eq:1d_isothermal}
		\end{equation}
		which corresponds to taking the gravitational acceleration $g(x)=-\phi'(x)$. We set these initial conditions in the domain $\Omega=[0,1]$ for $\phi(x)=10 x$ with Dirichlet boundary conditions and $\phi(x)=\sin(2\pi x)$ with periodic boundary conditions. Dirichlet boundary conditions are realized via constant-in-time ghost-cells, which are initialized with the exact solution (\ie \cref{eq:1d_isothermal} in this case). 
		For this isothermal solution the speed of sound is $c=\sqrt{\gamma p/\rho}=\sqrt{\gamma}.$ The sound crossing time is defined as
		\begin{equation}
		\label{eq:soundcrossing_time}
		\tau:=\int_\Omega\frac1c dx
		\end{equation}
		which yields $\tau=\sqrt{1/\gamma}$ and we run the test up to final time $t=2\tau\approx1.7$.
		Convergence rates for the standard method and the proposed well-balanced methods can be seen in \cref{tab:1d_isothermal_O3,tab:1d_isothermal_O3_periodic} for the energy density. The density and momentum show the same trends and are not shown for brevity. The standard method shows convergence rates as expected. The well-balanced methods are not only more accurate than the standard method, they also show better convergence rate. Using the \dwb method increases the order of accuracy by one order, using the \la method increases it by two orders.
		
		The initial condition used in the above tests does not satisfy the discrete hydrostatic conditions in the theorem, so the errors are not exactly (machine) zero. We now apply the \dwb method to a slightly modified setup; we use the density given in \cref{eq:1d_isothermal} and integrate the internal energy such that it satisfies \cref{eq:wbcondition}. This initial data are a third or fifth order accurate discretization of \cref{eq:1d_isothermal} respectively. \refb{We apply the three different boundary conditions introduced in \cref{sec:dwb_bcs} and the errors can be seen in \cref{tab:1d_isothermal_discrete}.} The \dwb method maintains the discretized hydrostatic solution to machine precision. This is valid for the third as well as the fifth order methods \refb{and for all boundary conditions}. 
		
		\begin{table}
			\centering
			\caption{
				\label{tab:1d_isothermal_O3}
				$L^1$-errors and rates in total energy for an isothermal hydrostatic solution of the Euler equations after two sound crossing times computed using different methods. The setup is described in \cref{sec:1d_isothermal}, the gravitational potential is $\phi(x)=10 x$. 
			}
			\setlength\tabcolsep{3.3pt}
			\begin{tabular}{| c | c c | c c | c c | c c | c c | c c |}
				\hline
				\multirow{2}{*}{$N$} & \multicolumn{2}{c|}{Std-O3} & \multicolumn{2}{c|}{DWB-O3} & \multicolumn{2}{c|}{LA-O3} & \multicolumn{2}{c|}{Std-O5} & \multicolumn{2}{c|}{DWB-O5} & \multicolumn{2}{c|}{LA-O5} \\
				& $E$ error & rate& $E$ error & rate& $E$ error & rate& $E$ error & rate& $E$ error & rate& $E$ error & rate\\
				\hline
				128  & 1.07e-04 & \multirow{2}{*}{3.0} & 2.03e-07 & \multirow{2}{*}{4.0} & 1.65e-06 & \multirow{2}{*}{5.1} &  3.19e-07 & \multirow{2}{*}{5.0} & 6.59e-10 & \multirow{2}{*}{6.0} & 6.03e-10 & \multirow{2}{*}{6.5} \\
				256  & 1.29e-05 & \multirow{2}{*}{3.0} & 1.23e-08 & \multirow{2}{*}{4.0} & 4.95e-08 & \multirow{2}{*}{5.1}   & 1.01e-08 & \multirow{2}{*}{5.0} & 1.03e-11 & \multirow{2}{*}{6.1} & 6.60e-12 & \multirow{2}{*}{7.0} \\
				512  & 1.59e-06 &                      & 7.60e-10 &                      & 1.42e-09 &                       & 3.14e-10 &                      & 1.52e-13 &                      & 5.11e-14 &                      \\
				\hline
			\end{tabular}
			
			
		\end{table}
		\begin{table}
			\centering
			\caption{
				\label{tab:1d_isothermal_O3_periodic}
				$L^1$-errors and rates in total energy for an isothermal hydrostatic solution of the Euler equations after two sound crossing times computed using different methods. The setup is described in \cref{sec:1d_isothermal}, the gravitational potential is $\phi(x)=\sin(2\pi x)$.
			}
			\setlength\tabcolsep{3.3pt}
			\begin{tabular}{| c | c c | c c | c c | c c | c c | c c |}
				\hline
				\multirow{2}{*}{$N$} & \multicolumn{2}{c|}{Std-O3} & \multicolumn{2}{c|}{DWB-O3} & \multicolumn{2}{c|}{LA-O3} & \multicolumn{2}{c|}{Std-O5} & \multicolumn{2}{c|}{DWB-O5} & \multicolumn{2}{c|}{LA-O5} \\
				& $E$ error & rate& $E$ error & rate& $E$ error & rate& $E$ error & rate& $E$ error & rate& $E$ error & rate\\
				\hline
				128  & 3.00e-04 & \multirow{2}{*}{3.0} & 1.48e-07 & \multirow{2}{*}{4.0} & 7.02e-07 & \multirow{2}{*}{4.8} & 2.55e-07 & \multirow{2}{*}{5.0} & 5.45e-10 & \multirow{2}{*}{6.0} & 5.15e-10 & \multirow{2}{*}{6.9} \\
				256  & 3.64e-05 & \multirow{2}{*}{3.0} & 9.14e-09 & \multirow{2}{*}{4.0} & 2.50e-08 & \multirow{2}{*}{4.8} & 8.07e-09 & \multirow{2}{*}{5.0} & 8.61e-12 & \multirow{2}{*}{5.4} & 4.36e-12 & \multirow{2}{*}{4.3} \\
				512  & 4.65e-06 &                      & 5.55e-10 &                      & 9.14e-10 &                      & 2.53e-10 &                      & 2.01e-13 &                      & 2.17e-13 &                      \\
				\hline
			\end{tabular}
			
			%
			%
		\end{table}
		%
		%
		\begin{table}
			\centering
			\caption{
				\label{tab:1d_isothermal_discrete}
				$L^1$-errors for an discrete isothermal hydrostatic solution of the Euler equations after two sound crossing times computed using different methods with a resolution of 128 cells. The initial data satisfy \cref{eq:wbcondition}. The setup is described in \cref{sec:1d_isothermal}.
			}\refb{
			\begin{tabular}{|c |c |c |c |c |}
				\hline
				Method & boundary condition & $\rho$ error & $\rho u$ error  & $E$ error  \\
				\hline
				& Dirichlet & 1.28e-16 & 7.64e-17 & 7.44e-16 \\
				\dwb-O3 & hydrostatic extrapolation & 7.62e-16 &  5.52e-16 & 3.05e-15 \\
				& solid wall &  7.62e-16 &  5.52e-16 & 3.05e-15 \\
				\hline
				& Dirichlet & 1.33e-16 & 1.57e-16 & 1.67e-16 \\
				\dwb-O5 & hydrostatic extrapolation & 4.30e-16 & 3.15e-16 &1.20e-16   \\
				& solid wall & 4.30e-16 & 3.15e-16 &1.20e-16 \\
				\hline
			\end{tabular}}
		\end{table}

		
		\subsection{1-d hydrostatic solution with perturbation}
		\label{sec:1d_isothermal_pert}
		
		Now we use the periodic potential $\phi(x)=\sin(2\pi x)$ and $g(x)=-\phi'(x)$ for the isothermal solution and add a perturbation 
		\begin{align}
		\rho(x)&=\tilde{\rho}(x),& 
		u(x)   &=\tilde u(x), & 
		p(x)   &=\tilde{p}(x)+\eta \exp\left(-100\left(x-\half\right)^2\right)
		\end{align}
		with $\eta=10^{-1}$. We compute this test up to time $t=0.5$. We compare the results with a simulation obtained from a seventh order standard method on a grid of $2024$ cells. The errors and convergence rates in total energy of the standard and well-balanced methods are shown in \cref{tab:1d_isothermal_pert}. Since they show the same trend, we omit showing errors in density and momentum for brevity. All methods show rates close to the expected rates for third and fifth order convergence, respectively.
		\begin{table}
			\centering
			\caption{
				\label{tab:1d_isothermal_pert}
				$L^1$-errors and rates in total energy for the isothermal hydrostatic solution with perturbation $\eta=10^{-1}$ of the Euler equations after time $t=0.5$ computed using different methods. The setup is described in \cref{sec:1d_isothermal_pert}.
			}
			\setlength\tabcolsep{3.3pt}
			\begin{tabular}{| c | c c | c c | c c | c c | c c | c c | c c | c c | c c | c c | c c | c c |}
				\hline
				\multirow{2}{*}{$N$} & \multicolumn{2}{c|}{Std-O3} & \multicolumn{2}{c|}{DWB-O3} & \multicolumn{2}{c|}{LA-O3} & \multicolumn{2}{c|}{Std-O5} & \multicolumn{2}{c|}{DWB-O5} & \multicolumn{2}{c|}{LA-O5} \\
				& $E$ error & rate& $E$ error & rate& $E$ error & rate& $E$ error & rate& $E$ error & rate& $E$ error & rate\\
				\hline
				128  & 5.73e-04 & \multirow{2}{*}{2.6} & 7.59e-04 & \multirow{2}{*}{2.6} & 7.55e-04 & \multirow{2}{*}{2.6} & 1.16e-05 & \multirow{2}{*}{4.7} & 1.59e-05 & \multirow{2}{*}{4.7} & 1.59e-05 & \multirow{2}{*}{4.7} \\
				256  & 9.78e-05 & \multirow{2}{*}{2.5} & 1.26e-04 & \multirow{2}{*}{2.8} & 1.25e-04 & \multirow{2}{*}{2.8} & 4.49e-07 & \multirow{2}{*}{5.0} & 6.33e-07 & \multirow{2}{*}{4.8} & 6.33e-07 & \multirow{2}{*}{4.8} \\
				512  & 1.73e-05 &                      & 1.80e-05 &                      & 1.80e-05 &                      & 1.42e-08 &                      & 2.20e-08 &                      & 2.20e-08 &                      \\
				\hline
			\end{tabular}
			%
			%
			%
			%
		\end{table}
		\begin{table}
			\centering
			\caption{
				\label{tab:1d_isothermal_pert_1e-5}
				$L^1$-errors and rates in total energy for the isothermal hydrostatic solution with perturbation $\eta=10^{-5}$ of the Euler equations after time $t=0.5$ computed using different methods. The setup is described in \cref{sec:1d_isothermal_pert}.
			}
			\setlength\tabcolsep{3.3pt}
			\begin{tabular}{| c | c c | c c | c c | c c | c c | c c | c c | c c | c c | c c | c c | c c |}
				\hline
				\multirow{2}{*}{$N$} & \multicolumn{2}{c|}{Std-O3} & \multicolumn{2}{c|}{DWB-O3} & \multicolumn{2}{c|}{LA-O3} & \multicolumn{2}{c|}{Std-O5} & \multicolumn{2}{c|}{DWB-O5} & \multicolumn{2}{c|}{LA-O5} \\
				& $E$ error & rate& $E$ error & rate& $E$ error & rate& $E$ error & rate& $E$ error & rate& $E$ error & rate\\
				\hline
				128  & 2.05e-04 & \multirow{2}{*}{2.6} & 6.51e-07 & \multirow{2}{*}{4.0} & 5.47e-07 & \multirow{2}{*}{4.5} & 1.38e-07 & \multirow{2}{*}{4.9} & 2.10e-09 & \multirow{2}{*}{5.8} & 1.24e-09 & \multirow{2}{*}{5.6} \\
				256  & 3.43e-05 & \multirow{2}{*}{2.7} & 4.06e-08 & \multirow{2}{*}{4.0} & 2.49e-08 & \multirow{2}{*}{4.0} & 4.51e-09 & \multirow{2}{*}{5.0} & 3.79e-11 & \multirow{2}{*}{5.4} & 2.61e-11 & \multirow{2}{*}{5.2} \\
				512  & 5.16e-06 &                      & 2.55e-09 &                      & 1.56e-09 &                       & 1.43e-10 &                      & 8.85e-13 &                      & 7.31e-13 &                      \\
				\hline
			\end{tabular}
			%
			%
			%
			%
		\end{table}
		Next, to illustrate the capability of the well-balanced methods to capture small perturbations on the hydrostatic solution on a coarse grid, we use different grid sizes and methods for a perturbation amplitude of $\eta=10^{-5}$. The corresponding energy errors and rates at time $t=0.5$ are presented in \cref{tab:1d_isothermal_pert_1e-5} and the density perturbations from the hydrostatic state are visualized in in \cref{fig:1d_isothermal_pert_O3,fig:1d_isothermal_pert_O5}.
		\begin{figure}
			\centering
			\includegraphics[scale=1]{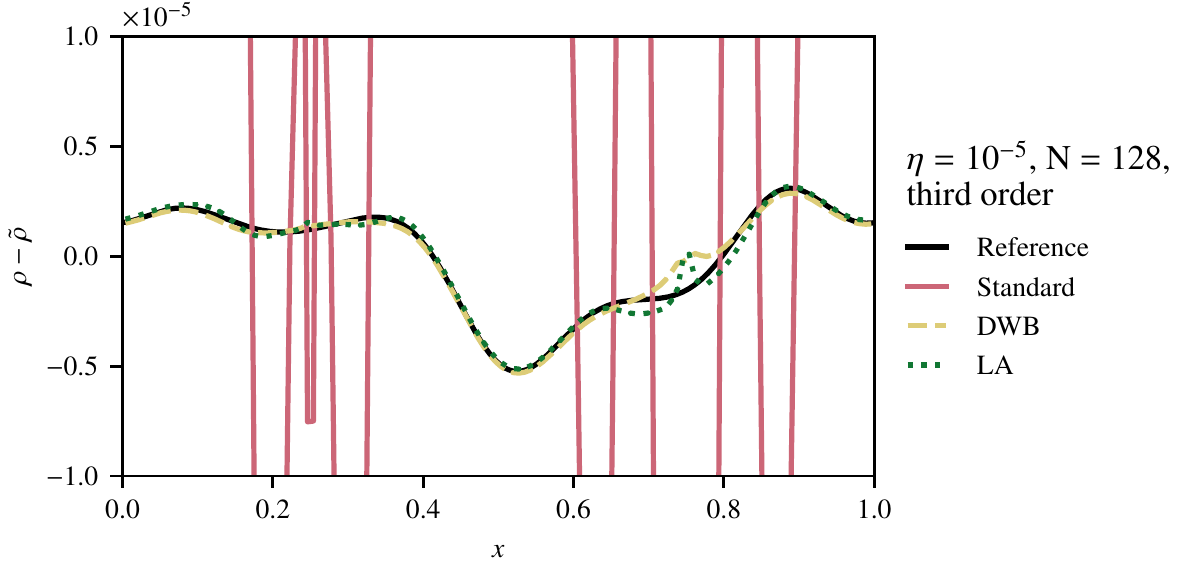}\\
			\includegraphics[scale=1]{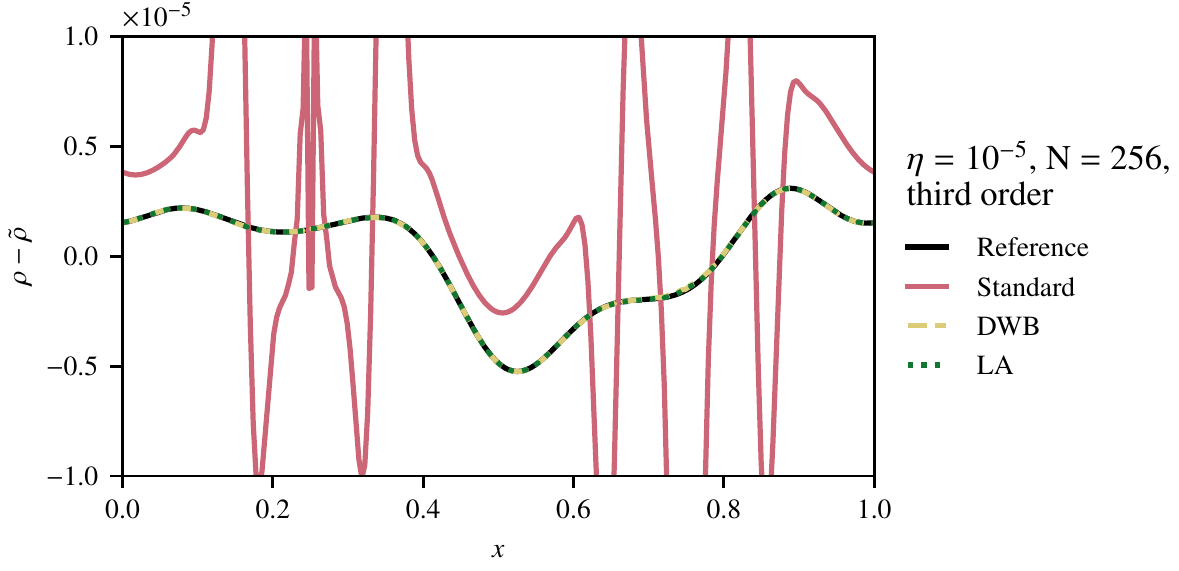}\\
			\caption{\label{fig:1d_isothermal_pert_O3} Density deviation from the hydrostatic background for $\eta=10^{-5}$ using different grids at $t=0.5$. Third order methods have been used. The test setup is described in \cref{sec:1d_isothermal_pert}.
			}
		\end{figure}
		\begin{figure}
			\centering
			\includegraphics[scale=1]{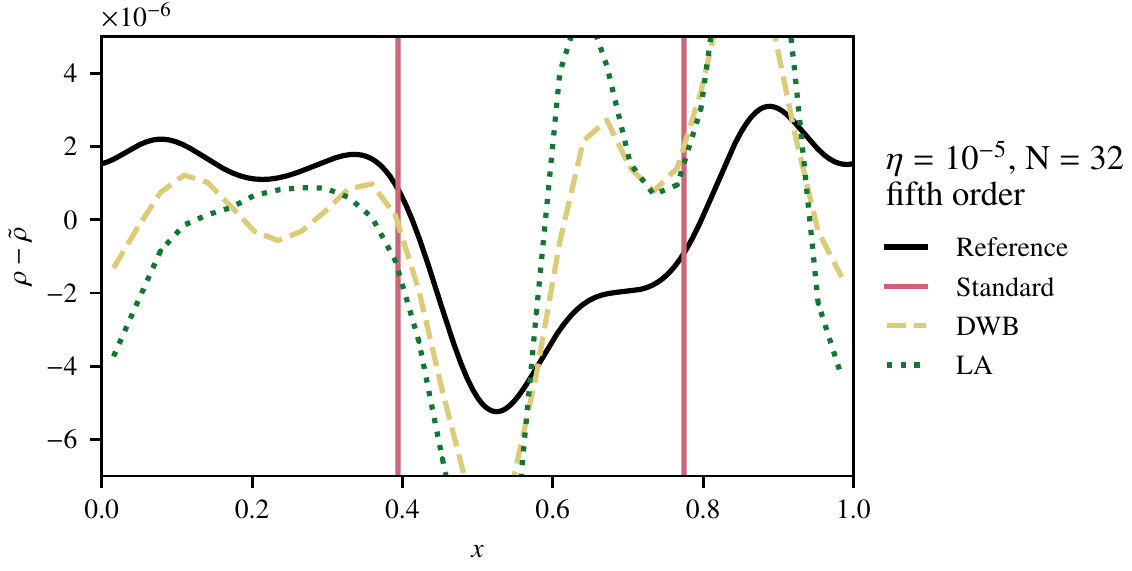}\\
			\includegraphics[scale=1]{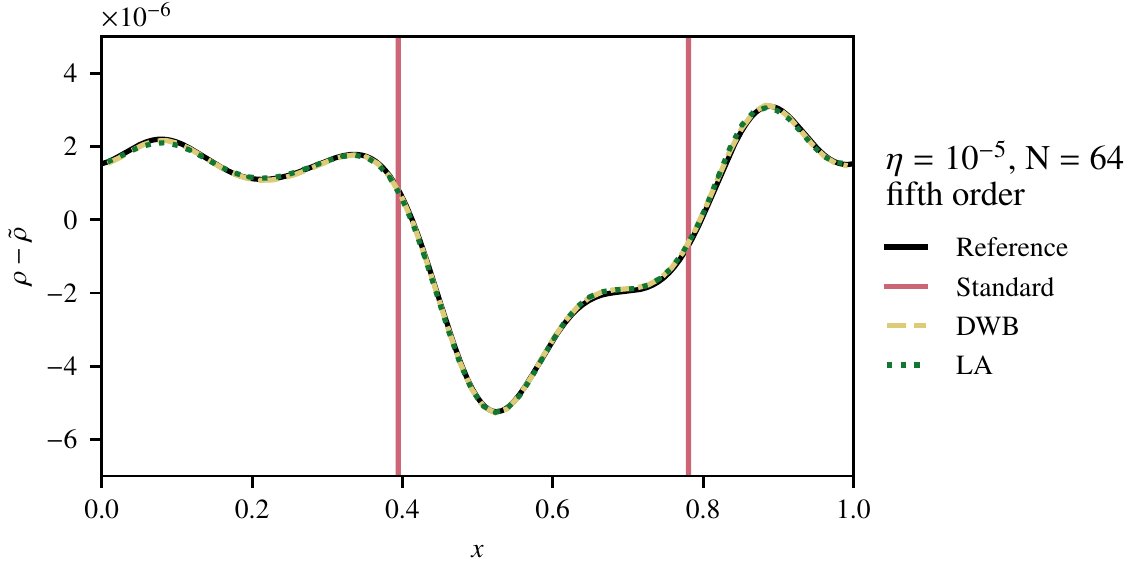}\\
			\caption{\label{fig:1d_isothermal_pert_O5} Density deviation from the hydrostatic background for $\eta=10^{-5}$ using different grids at $t=0.5$. Fifth order methods have been used. The test setup is described in \cref{sec:1d_isothermal_pert}.
			}
		\end{figure}
		The well-balanced methods succeed in resolving the perturbation significantly more accurately than the non-well-balanced standard method; the errors from the standard method are so large that they are not completely visible in the figures. The \la method is able to capture the small perturbations as accurately as the \dwb well-balanced method.
		
		\subsection{Hydrostatic solution for a non-ideal gas equation of state}
		\label{sec:1d_non_ideal}
		\subsubsection{Polytropic hydrostatic solution}
		\label{sec:non_ideal_polytropic}
		Polytropic solutions of \cref{eq:hystat} are given by
		\begin{align}
		\theta(x)&:=1-\frac{\nu-1}{\nu}\phi(x),&\tilde\rho(x) &= \theta(x)^{\frac{1}{\nu-1}},&\tilde p(x)&:=\tilde\rho(x)^\nu,
		\label{eq:polytropic}
		\end{align}
		and $u=0$. \Cref{eq:polytropic} describes a static state of the compressible Euler equations independent from the EoS. We choose the equation of state for an ideal gas which is additionally subject to radiation pressure \cite{Chandrasekhar1958}
		\begin{equation}
		p = \rho T + T^4,
		\label{eq:radiative_eos}
		\end{equation}
		where the temperature $T$ is defined implicitly via
		\begin{equation}
		\varepsilon = \frac{\rho T}{\gamma-1} + 3 T^4.
		\label{eq:radiative_epsilon}
		\end{equation}
		The conversion between pressure $p$ and internal energy density $\varepsilon$ (while knowing the density $\rho$) cannot be computed explicitly. Instead, we use Newton's method to convert between $p$ and $\varepsilon$. The speed of sound for this EoS can be computed by \begin{align}
		\label{eq:c_radpres}
		c&=\sqrt{\frac {\Gamma_1p} \rho} &
		&\text{where}&
		\Gamma_1 &= \beta + \frac{(4-3\beta)^2(\gamma-1)}{\beta+12(\gamma-1)(1-\beta)}&
		&\text{with}&
		\beta &= \frac{\rho T}{p}.
		\end{align}
		The speed of sound is also computed using Newton's method. As for the ideal gas we use $\gamma=1.4$.
		For the gravity potential we choose $\phi(x)=gx$ with constant $g=-1$. The domain is $\Omega=[0,1]$ and Dirichlet boundary conditions are applied. 
		The sound crossing time for this setup, computed from \cref{eq:soundcrossing_time} and \cref{eq:c_radpres}, is $\tau\approx0.7$. We run the test to a final time of $t=10\approx14\tau$.
		We use a standard method and the extensions of the \dwb and \la methods for general EoS as described in \cref{sec:discrete_wb_eos}. For the \dwb method we use the iterative (\dwbeos) and the fast (\dwbeosf) computation of the cell-centered pressure. For the \la method we only use the fast computation (\laeosf) since the \la method is an approximately well-balanced method anyway.
		
		The $L^1$-errors and convergence rates in total energy \refb{with respect to the initial stratification} are shown in \cref{tab:non_ideal_polytropic}. Using the well-balanced methods significantly reduces the errors (about one to two orders of magnitude). Note, that the difference between the two different versions of the \dwb method is small. This justifies the usage of the simplified and much faster computation of $p_0$. The \laeosf method performs best. For the general EoS we do not observe the increased order of convergence that was observed in \cref{sec:1d_isothermal} in the case of the ideal gas EoS.
		\begin{table}
			\centering
			\caption{%
				\label{tab:non_ideal_polytropic}%
				$L^1$-errors and rates in total energy for the polytropic hydrostatic solution with an EoS for ideal gas with radiation pressure at time $t=10$ computed using different third order accurate methods. The setup is described in \cref{sec:non_ideal_polytropic}.
			}
			
			\refb{%
			\begin{tabular}{| c | c c | c c | c c | c c |}
				\hline
				\multirow{2}{*}{$N$} & \multicolumn{2}{c|}{Std-O3} & \multicolumn{2}{c|}{DWB-O3} & \multicolumn{2}{c|}{\dwbeosf-O3} & \multicolumn{2}{c|}{\laeosf-O3} \\
				& $E$ error & rate& $E$ error & rate& $E$ error & rate& $E$ error & rate\\
				\hline
				16  & 1.01e-05 & \multirow{2}{*}{3.2} & 4.85e-07 & \multirow{2}{*}{3.2} & 4.83e-07 & \multirow{2}{*}{3.2} & 4.27e-07 & \multirow{2}{*}{3.2} \\
				32  & 1.10e-06 & \multirow{2}{*}{3.3} & 5.48e-08 & \multirow{2}{*}{3.3} & 5.42e-08 & \multirow{2}{*}{3.3} & 4.71e-08 & \multirow{2}{*}{3.3} \\
				64  & 1.08e-07 &                      & 5.65e-09 &                      & 5.65e-09 &                      & 4.69e-09 &                      \\
				\hline
			\end{tabular}%
			}
		\end{table}
		
		\subsubsection{Perturbation on a polytropic hydrostatic solution}
		\label{sec:non_ideal_polytropic_perturbation}
		
		Now we add a Gaussian perturbation to the hydrostatic solution and use the initial data
		\begin{align*}
		\rho(x)&=\tilde{\rho}(x),& 
		p(x)   &=\tilde{p}(x)+\eta \exp\left(-100\left(x-0.3\right)^2\right),
		\end{align*}
		and $u=0$
		with  $\eta=10^{-7}$. We evolve these initial data up to time $t=0.1$ with $128$ grid cells resolution. Third order standard and well-balanced methods are applied. In the well-balanced methods we use the simplified way of computing the cell-centered pressure as described in \cref{sec:discrete_wb_eos} (\dwbeosf, \laeosf). The result is shown in \cref{fig:non_ideal_polytropic_perturbation}. The reference solution is computed using the seventh order accurate exactly well-balanced method from \cite{Berberich2020} on a grid of 2048 cells.
		\begin{figure}
			\centering
			\includegraphics[scale=1]{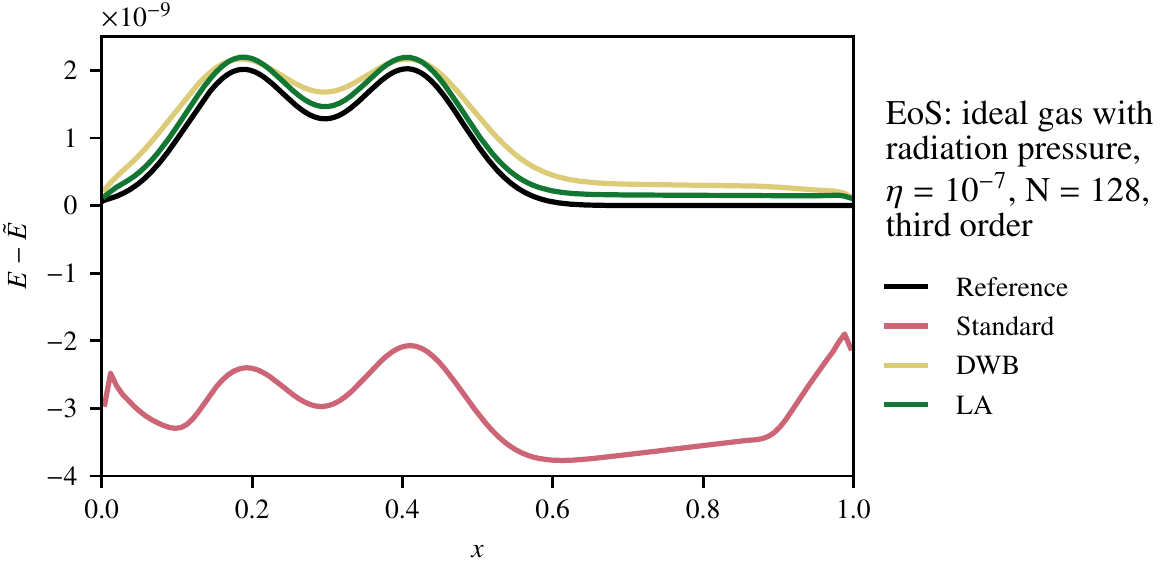}
			\caption{\label{fig:non_ideal_polytropic_perturbation}%
				Small perturbation ($\eta=10^{-7}$) on a polytropic hydrostatic solution with an EoS for ideal gas with radiation pressure after time $t=0.1$ computed using different third order accurate methods. The setup is described in \cref{sec:non_ideal_polytropic_perturbation}. Energy perturbations $E-\tilde{E}$ are shown.
			}
		\end{figure}
		Both well-balanced methods yield results much closer to the reference solution compared to the standard method. It is evident from these results that the approximate well-balanced methods can help resolve small perturbations to hydrostatic states more accurately even with an equation of state different from ideal gas. 
		
		\subsection{Riemann problem on an isothermal hydrostatic state}
		\label{sec:robustness}

		In this test we use the initial data
		\begin{align}
		\label{eq:hystatWithRiemann_rho}
		\tilde\rho(x)&:=\left\{\begin{matrix}
		a c\exp(-a\phi(x)) & \text{if $x<x_0$},\\
		b\exp(-b\phi(x)) & \text{if $x\geq x_0$},
		\end{matrix}\right.\\
		\label{eq:hystatWithRiemann_p}
		\tilde p(x)&:=\left\{\begin{matrix}
		c\exp(-a\phi(x)) & \text{if $x<x_0$},\\
		\exp(-b\phi(x)) & \text{if $x\geq x_0$},
		\end{matrix}\right.
		\end{align}
		$g(x) = -\phi'(x)$,
		and \crefrange{eq:hystatWithRiemann_rho}{eq:hystatWithRiemann_p} describe a piecewise isothermal hydrostatic solution with a jump discontinuity at $x=x_0$, which gives rise to all three waves of the Euler equations; the parameters are chosen as $x_0=0.125, a=0.5, b=1, c=2$. An ideal gas EoS with $\gamma=1.4$ is applied. We set these initial data on the domain $[0,0.25]$ and evolve them to the final time $t=0.02$ using our third and fifth order methods on a grid with $128$ cells and Dirichlet boundary conditions. As a reference solution to compute the error we use a numerical solution obtained using a standard first order method with $32768$ cells. In \cref{fig:isothermal_riemann}, we see the numerical results at final time for the \la methods. No spurious oscillations are visible. Using the \dwb method leads to very similar results, hence we omit showing them for brevity.
		
		To give quantitative results, we also compute the total variation of the solution at final time for all methods. The total variation of a quantity $\alpha=\rho,\rho u,E$ of a numerical solution is defined by
		$$
		\mathrm{TV}(\alpha):=\sum_{i=1}^N |\alpha_{i}-\alpha_{i-1}|.
		$$
		In \cref{tab:isotermal_riemann_tv} we present the difference in total variation relative to the total variation of the reference solution
		\begin{equation}
		\theta(\alpha):= \frac{\mathrm{TV}(\alpha)}{\mathrm{TV}(\alpha_\text{ref})}-1.
		\label{eq:osciallation_indicator}
		\end{equation}
		A negative value of $\theta$ indicates, that the total variation is smaller than in the reference solution. A positive value of $\theta$ means that there are additional oscillations.
		In \cref{tab:isotermal_riemann_tv}, the $\theta$ values for different methods are presented together with the $L^1$ errors. All methods lead to a decrease in total variation in conserved variables. \refb{Note that there are small visible oscillations if the CWENO5 scheme is used. However, this is common for CWENO methods since they are not exactly TVD.}
		
		\begin{figure}
			\centering
			\includegraphics[scale=1.]{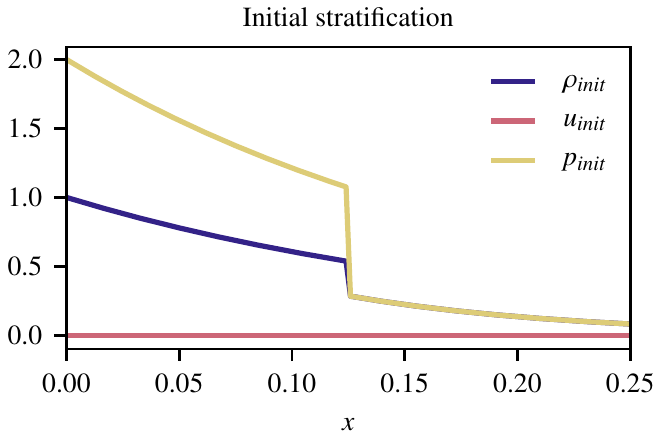}\hfill
			\includegraphics[scale=1.]{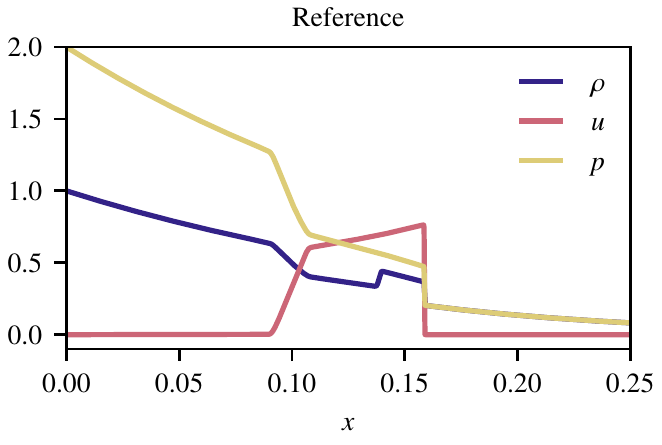}\\[.5em]
			\includegraphics[scale=1.]{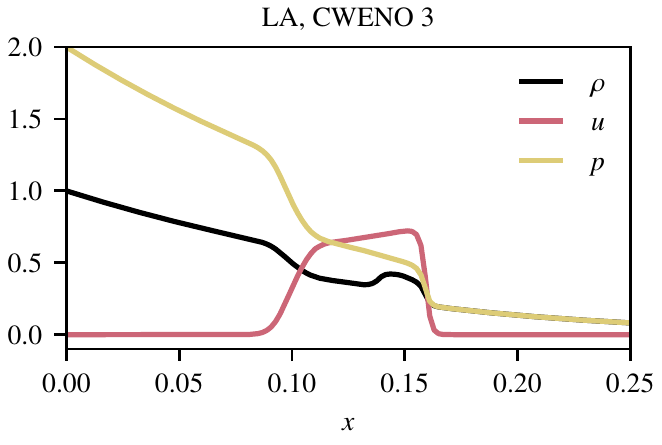}\hfill
			\includegraphics[scale=1.]{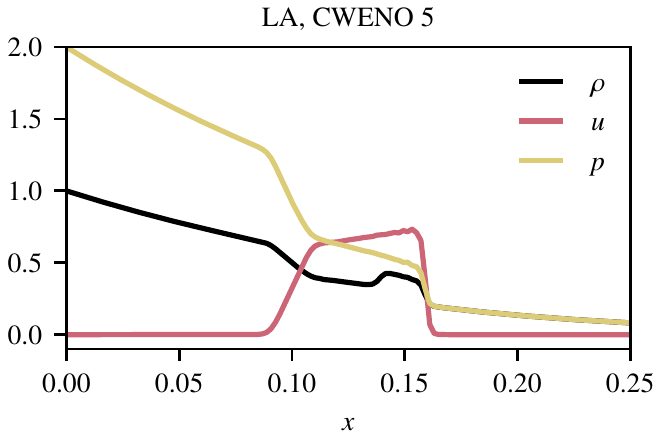}
			\caption{Initial data (top left panel), reference solution (top right panel), and simulation results (bottom panels) for the tests performed in \cref {sec:robustness}. The formally third (bottom left) and fifth (bottom right) order \la methods are used.  \label{fig:isothermal_riemann}}
		\end{figure}
		\begin{table}
			\centering
			\caption{
				\label{tab:isotermal_riemann_tv} Errors and total variation for the robustness test from \cref {sec:robustness} at final time $t=0.02$. The formally third and fifth order standard and well-balanced methods (\dwb and \la) are used. The oscillation indicator $\theta$ is defined in \cref{eq:osciallation_indicator}. \label{tab:isothermal_riemann} 
			}
			\begin{tabular}{|c |c c| c c |c c|}
				\hline
				Method  & $\rho$ error & $\theta(\rho)$ & $\rho u$ error & $\theta(\rho u)$ & $E$ error & $\theta(E)$\\
				\hline
				Std-O3&   6.51e-04 & -4.18e-02 & 8.30e-04 & -2.76e-02 & 3.25e-03 & -5.73e-06\\
				\dwb-O3 &  6.57e-04 & -3.95e-02 & 8.33e-04 & -2.70e-02 & 3.28e-03 & -5.74e-06\\
				\la-O3 &  6.56e-04 & -3.96e-02 & 8.33e-04 & -2.70e-02 & 3.28e-03 & -5.74e-06\\
				\hline
				Std-O5&  5.00e-04 & -3.67e-02 & 6.82e-04 & -2.50e-02 & 2.79e-03 & -5.74e-06\\
				\dwb-O5 &  4.76e-04 & -3.35e-02 & 6.48e-04 & -2.35e-02 & 2.71e-03 & -5.74e-06\\
				\la-O5 &  4.76e-04 & -3.36e-02 & 6.47e-04 & -2.38e-02 & 2.71e-03 & -5.74e-06\\
				\hline
			\end{tabular}
		\end{table}
		
		\refb{
		\subsection{Relaxation towards an unknown
hydrostatic state}
		\label{sec:non-barotropic}
		All the previous test cases have been based on isothermal or polytropic hydrostatic solutions, i.e.\ hydrostatic solution for which the majority of the well-balanced methods in literature can be applied. In this section we aim to show that the proposed methods can also be applied in cases in which the hydrostatic solution is not isentropic, polytropic, or isothermal and is also not known a priori. This ensures that the test case is suitable for neither well-balancing methods of type 1 (\cref{def:wb_type1}) nor type 2 (\cref{def:wb_type2}).
		
		For this purpose, we start with initial conditions that are \emph{not} in hydrostatic state and let the solution evolve using Euler equations with gravitational source term, an ideal gas EoS, and  momentum damping as described below until it reaches a stationary state. We choose
		\begin{align}
			\phi(x)&:= 10 \sin(2\pi x),& 
			\rho(t=0,x)&:= 1,& 
			p(t=0,x)&:= 1,& 
			u(t=0,x)&:= 0
		\end{align}
		on the domain $[0,1]$ with periodic boundary conditions. In order to drive the solution to a static state, we apply the simple momentum damping operator that implements the damping 
		$$
			\tdt(\rho u)=-\delta(\rho u)
		$$ 
		in a operator-split fashion in each Runge--Kutta sub-step with the damping-coefficient $\delta=0.2$.
		We apply the standard, \dwb, and \la method on 128 and 256 cells. The maximal velocity on the grid over time is shown in the top panel of \cref{fig:non-barotropic}. It gets evident that the standard method fails to settle to a static configuration. Using the \la method significantly reduces the velocities in the final configuration and under refinement of the grid the improved convergence behavior gets evident once more. Using the \dwb method allows the simulation to settle to a static stratification with velocities of the size of the machine error. The bottom panel shows pressure over density in a log-log diagram for the final stratification obtained using the different methods with different resolutions. It gets evident that density and pressure can not be related by a power law which shows that the stratification is not isothermal, isentropic, or polytropic. Furthermore, the figure shows that the well-balanced methods significantly reduce the grid-dependency of the final stationary state.
		\begin{figure}
			\centering
			\includegraphics[scale=1]{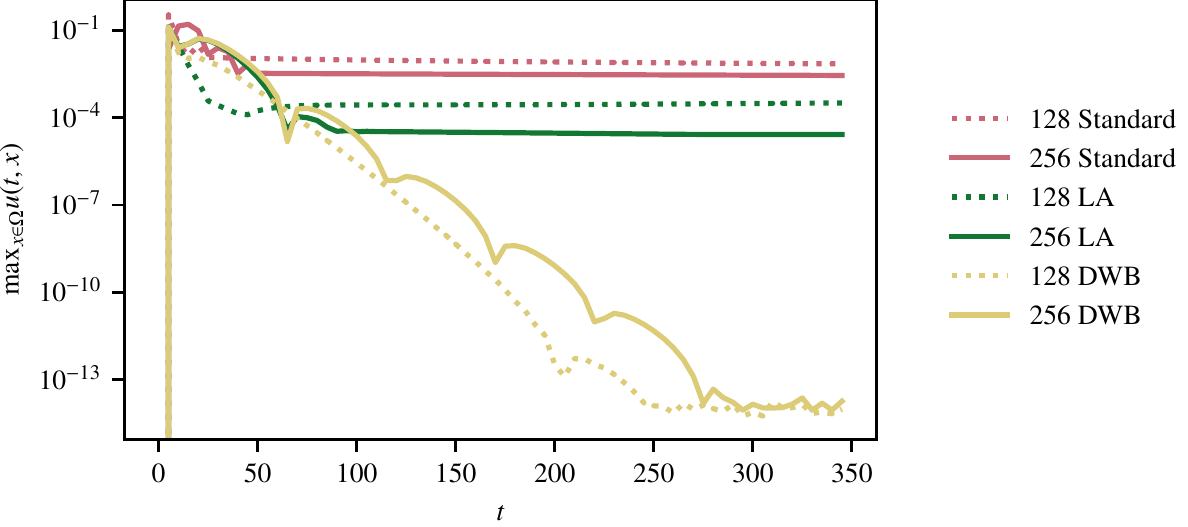}
			\\
			\includegraphics[scale=1]{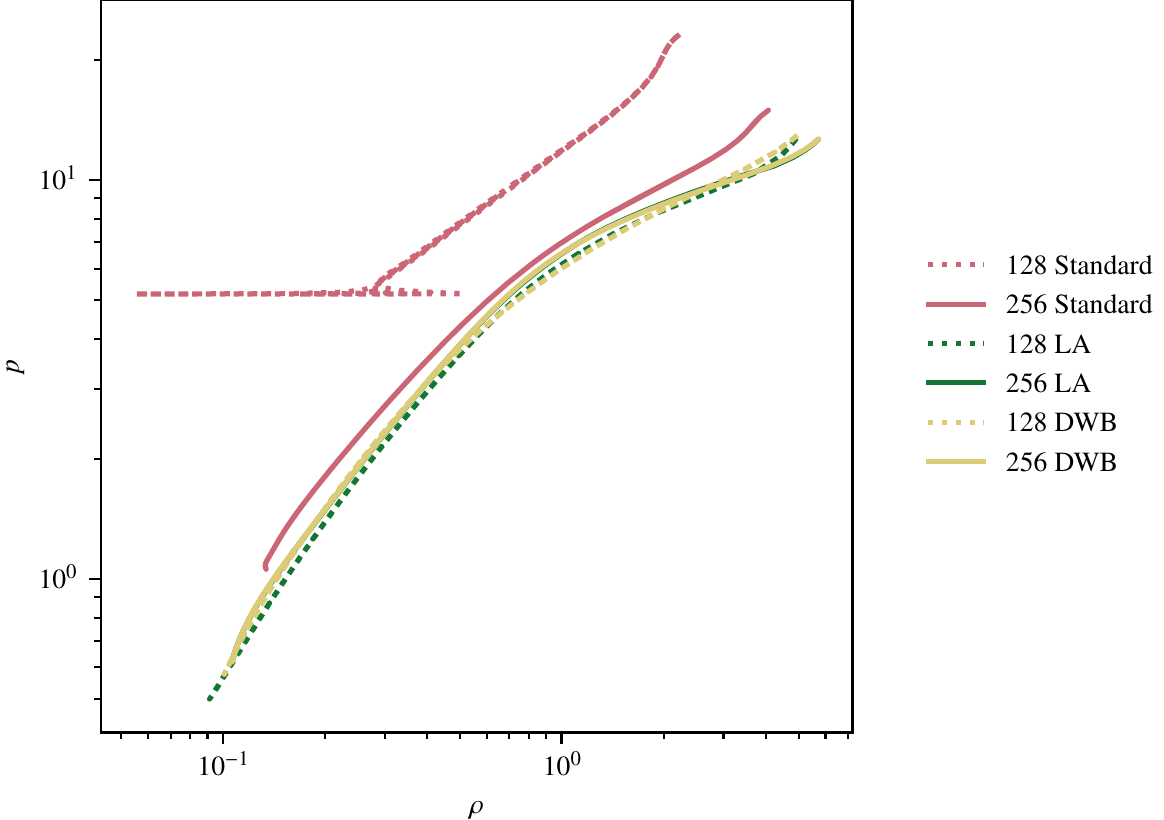}
			\caption{\label{fig:non-barotropic}\refb{%
					Data resulting from the test case described in \cref{sec:non-barotropic}. The stratification is driven to a hydrostatic state by damping. Top: Maximal velocity over time. The standard method never settles at hydrostatic state since the velocities never vanish. The \la method settles at a significantly lower level while the \dwb method is damped to a discrete static state with velocities on machine error. Bottom: Pressure over density in the final stratification in a log-log diagram. This verifies that there is no isothermal, isentropic, or polytropic relation in the hydrostatic state that is finally achieved in this test case.}}
		\end{figure}

		}
		\section{Extension to two spatial dimensions}\label{sec:2d}
		\subsection{Two-dimensional compressible Euler equations with gravity}\label{sec:eul_2d}
		The two-dimensional compressible Euler equations which model the balance laws of mass, momentum, and energy under the influence of gravity are given by
		\begin{equation}
		\dt{\q}+ \dx{\vec f_x} + \dy{\vec f_y} = \s,
		\label{eq:eul2d_cartesian}
		\end{equation}
		where the conserved variables, fluxes and source terms are
		\begin{align}
		\label{eq:euler2d}
		\q &= \begin{pmatrix}
		\rho \\ \rho u \\ \rho v \\ E \end{pmatrix}, &
		\f_x &= \begin{pmatrix}
		\rho u \\
		p + \rho u^2 \\
		\rho u v \\
		(E+p)u \end{pmatrix}, & 
		\f_y &= \begin{pmatrix}
		\rho v \\
		\rho u v \\
		p + \rho v^2 \\
		(E+p)v \end{pmatrix}, &
		\s &= \begin{pmatrix}
		0 \\
		\rho \gx\\
		\rho \gy\\
		\rho (u \gx +  v \gy)
		\end{pmatrix}
		\end{align}
		with $\rho,p>0$.
		Moreover, $E=\varepsilon + \tfrac{1}{2}\rho|\vec v|^2$ is the total energy density with the velocity $\vel=(u,v)^T$ and internal energy density $\epsilon$. The vector valued function $\x\mapsto\g(\x)=(\gx(\x),\gy(\x))^T=\nabla\phi(\x)$ with $\x:=(x,y)^T$ is a given gravitational field. As in the one-dimensional case the system is closed using an EoS, which relates $\rho$, $p$, and $\varepsilon$.
		\subsection{Finite volume method in two spatial dimensions}
		\label{sec:fv_2d}

		We discretize the spatial domain $\Omega:=\left[x_{\min},x_{\max}\right]\times\left[y_{\min},y_{\max}\right]$  into $N_x\times N_y$ finite control volumes $\Omega_{ij}:=\left[x_{i-\half},x_{i+\half}\right]\times\left[y_{j-\half},y_{j+\half}\right]$ for $i=1,\dots,N_x$ and $j=1,\dots,N_y$. The interface positions in the $x$ and $y$-direction are $x_{i+\half}:=x_{\min}+i\Delta x$ and $y_{j+\half}:=y_{\min}+j\Delta y$
		with $\Delta x=(x_{\max}-x_{\min})/N_x,\Delta y=(y_{\max}-y_{\min})/N_y$
		for $i=0,\dots,N_x$ and $j=0,\dots,N_y$. The cell-centered coordinates are given by $\x_{ij}=(x_i,y_j)^T=\left(\frac12\left(x_{i-\half}+x_{i+\half}\right),\frac12\left(y_{j-\half}+y_{j+\half}\right)\right)$. The compressible Euler equations with gravity (\cref{eq:eul2d_cartesian}) are integrated over each control volume and yields 
		\begin{equation}
		\dt \hat\Q_{ij} = -\frac1{|\Omega_{ij}|}
		\left[
		\F_{i+\half,j}-\F_{i-\half,j} + \F_{i,j+\half}-\F_{i,j-\half}
		\right] + \hat \S_{ij},
		\end{equation}
		where the interface fluxes are approximated by
		\begin{align}
		\F_{i+\half,j}&:=\mathcal{Q}_{y\in[y_{j-\half},y_{j+\half}]}\left( \mathcal F_x\left(\Q^\text{rec}_{ij}\left(x_{i+\half},y\right),\Q^\text{rec}_{i+1,j}\left(x_{i+\half},y\right) \right)\right)
		\approx \int_{y_{j-\half}}^{y_{j+\half}}\f_x(\q(x_{i+\half},y)) \,dy,\\
		\F_{i,j+\half}&:=\mathcal{Q}_{x\in[x_{i-\half},x_{i+\half}]}\left( \mathcal F_y\left(\Q^\text{rec}_{ij}\left(x,y_{j+\half}\right),\Q^\text{rec}_{i,j+1}\left(x,y_{j+\half}\right) \right)\right)
		\approx \int_{x_{i-\half}}^{x_{i+\half}}\f_y(\q(x,y_{j+\half})) \,dx.
		\end{align}
		Here, $\mathcal Q$ is an $m$-th order accurate quadrature rule, $\mathcal F_x$ and $\mathcal F_y$ are numerical fluxes as defined in \cref{subsubsec:stan}, i.e.\ they are consistent, Lipschitz-continuous in both arguments, and they satisfy the contact property. The reconstructed quantities $\Q^\text{rec}_{ij}$ are obtained from an $m$-th order accurate reconstruction, i.e.\ 
		$$
		\Q_{ij}^\text{rec}(\x):=\mathcal R\left(\x,\left\{\hat\Q_{kl}\right\}_{(k.l)\in\mathcal S_{ij}}\right),
		$$
		where $\mathcal S_{ij}$ is the set of all index tuples of cells in the stencil of the reconstruction.
		$\hat\S_{ij}$ is a consistent source term discretization. In practical application of the standard method (in \cref{sec:ne_2d}) we use the source term discretization which is introduced later in \cref{eq:sourceterm_discretization_2d}. The semi-discrete scheme is evolved in time using a sufficiently high order accurate Runge--Kutta method.
		\subsection{Local approximation method}\label{sec:wb_2d}
		In the tests of the one-dimensional methods, we saw that the method with reduced stencil (\la) converges better to hydrostatic states and gives more accurate results than the method with the large stencil (\dwb). Also, in two spatial dimensions, a relation like \cref{eq:wbcondition} can not be easily defined for our polynomial approximation of the source term, since the curve integral is in general path-dependent, unless the gravitational field is parallel to grid lines. This rules out the formulation of a well-balanced theorem for a genuinely two-dimensional method. Since in multi-dimensional simulations, the compactness of the stencil is usually even more important, we only extend the \la method to two spatial dimensions. These methods will not be exactly well-balanced and the numerical experiments will show if they are useful in practice.
		
		\subsubsection{Source term discretization}
		\label{sec:sourceterm_2d}
		Let us define the source term approximation
		\begin{align*}
		\s_{ij}(\x)
		:=&
		\begin{pmatrix}
		s^x_{ij}(\x)\\s^y_{ij}(\x)
		\end{pmatrix}
		:=
		\begin{pmatrix}
		\rho_{ij}^\text{rec}(\x)(\gx)_{ij}^\text{int}(\x)\\
		\rho_{ij}^\text{rec}(\x)(\gy)_{ij}^\text{int}(\x)
		\end{pmatrix},
		\\
		(\vel\cdot\s)_{ij}(x):=&(\rho u)_{ij}^\text{rec}(\x)(\gx)_{ij}^\text{int}(\x)+(\rho v)_{ij}^\text{rec}(\x)(\gy)_{ij}^\text{int}(\x),
		\end{align*}
		where $\rho_{ij}^\text{rec}$ and $(\rho\vel)_{ij}^\text{rec}:=\left((\rho u)_{ij}^\text{rec},(\rho v)_{ij}^\text{rec}\right)^T$ are $m$-th order accurate CWENO reconstruction polynomials in the $ij$-th cell. $\g_{ij}^\text{int}$ is an $m$-the order accurate interpolation polynomial from the cell-centered point values of $\g$; CWENO interpolation could be used if $\g$ is not smooth. 
		Due to the polynomial character of $\s_{ij}$ and $(\vel\cdot\s)_{ij}$ the source term integrals can be computed explicitly.
		
		The cell-averaged source term used in the finite volume method in the $ij$-th cell is hence computed as
		\begin{equation}
		\label{eq:sourceterm_discretization_2d}
		\hat{\vec S}_{ij} := \frac1{|\Omega_{ij}|}\int_{\Omega_{ij}}\begin{pmatrix}
		0\\s^x_{ij}(\x)\\s^y_{ij}(\x)\\(\vel\cdot\s)_{ij}(\x)
		\end{pmatrix} d\x.
		\end{equation}
		\subsubsection{Reconstruction}
		\label{sec:reconstruction_2d}
		
		We construct a local approximation to the hydrostatic pressure in the cell $\Omega_{ij}$. For that, we first define the local hydrostatic density $\rho_{ij}^\eq:=\rho_{ij}^\text{rec}$. To obtain the local hydrostatic pressure, we integrate the hydrostatic equation from the cell center to any point which yields
		\begin{equation}
		p_{ij}^\text{eq}(\x):=p^0_{ij}+\int_0^1\s_{ij}(\x_{ij}+(\x-\x_{ij})t)\cdot(\x-\x_{ij}) \,dt
		\label{eq:p0_relation1_2d}
		\end{equation}
		for the approximation in the cell $\Omega_{ij}$. The cell-centered pressure value $p^0_{ij}$ is determined by demanding
		\begin{equation}
		\frac1{|\Omega_{ij}|}\int_{\Omega_{ij}}\varepsilon_{ij}^\text{eq}(\x) \,d\x= \hat{\varepsilon}_{ij},
		\label{eq:p0_relation2_2d}
		\end{equation}
		where
		$$\varepsilon_{ij}^\text{eq}(\x):=\varepsilon_{\text{EoS}}\left(\rho_{ij}^\text{eq}(\x),p_{ij}^\text{eq}(\x)\right)$$is defined via the EoS, 
		and the cell-averaged internal energy density is computed using
		\begin{equation}
		\hat{\varepsilon}_{ij}:= \hat E_{ij} - \frac1{2\Delta x \Delta y}\mathcal Q_{ij}\left(\frac{\left((\rho u)^\rec_{ij}\right)^2+\left((\rho v)^\rec_{ij}\right)^2}{{\rho}^\rec_{ij}}\right).
		\label{eq:cell_averaged_internal_energy_2d}
		\end{equation}
		Note that the relation \cref{eq:cell_averaged_internal_energy_2d} is only second order accurate in general. However, on hydrostatic solutions it is exact since the momentum term vanishes in that case.
		Assuming an ideal gas, \cref{eq:p0_relation1_2d,eq:p0_relation2_2d} yield
		\begin{equation}
		\label{eq:p0_relation_ideal_2d}
		p_{ij}^0 =(\gamma-1) \hat{\varepsilon}_{ij} - \frac1{|\Omega_{ij}|}\int_{\Omega_{ij}}\int_0^1\s_{ij}(\x_{ij}+(\x-\x_{ij})t)\cdot(\x-\x_{ij}) \,dt\,d\x.
		\end{equation}
		Now that the pressure at cell center $p_{0,ij}$ is fixed, we have fully
		specified the high-order accurate representation of the equilibrium conserved
		variables in cell
		$\Omega_{ij}$:
		\begin{equation}
		\label{eq:nm_2d_wb_0090}
		\Q^{eq}_{ij}(\x) = \begin{pmatrix}
		\rho^{eq}_{ij}(\x)   \\
		0                \\
		0                \\
		\varepsilon^{eq}_{ij}(\x)
		\end{pmatrix}
		.
		\end{equation}
		Similarly, the equilibrium reconstruction of the primitive variables are
		given by
		\begin{equation}
		\label{eq:nm_2d_wb_0100}
		\W^{eq}_{ij}(\x) = \begin{pmatrix}
		\rho^{eq}_{ij}(\x) \\
		0              \\
		0              \\
		p^{eq}_{ij}(\x)
		\end{pmatrix}
		.
		\end{equation}
		We stress here that the equilibrium density reconstruction is simply the
		result provided by the standard reconstruction procedure $\mathcal{R}$.
		
		Next, we develop the high-order equilibrium preserving reconstruction procedure. To this end, \refb{as in e.g.\ \cite{LeVeque1998b}, \cite{Botta2004}}, we decompose in every cell the solution into an equilibrium and a (possibly large) perturbation part. The equilibrium part in cell $\Omega_{ij}$ is simply given by $\Q^{eq}_{ij}(x)$ of Eq. \eqref{eq:nm_2d_wb_0090} above. The perturbation part in cell $\Omega_{ij}$ is obtained by applying the standard reconstruction procedure $\mathcal{R}$ to the cell-averaged equilibrium perturbation,
		\begin{equation}
		\label{eq:nm_2d_wb_0110}
		\delta \Q_{ij}(\x)
		= \mathcal{R}
		\left(
		\x
		;
		\left\{  \hat\Q_{kl}
		- \frac1{|\Omega_{ij}|}\mathcal Q_{kl}\left(\Q^{eq}_{ij}\right) \right\}_{(k,l) \in \mathcal S_{ij}}
		\right),
		\end{equation}
		where $\mathcal Q_{ik}$ is an at least $m$-th order accurate two-dimensional quadrature rule approximating the integral over the cell $\Omega_{kl}$.
		We note that the cell average of the equilibrium perturbation in cell
		$\Omega_{kl}$ is obtained by taking the difference between the cell average
		$\hat\Q_{kl}$ and the cell average of the equilibrium $\Q^{eq}_{ij}$ in cell
		$\Omega_{kl}$.
		
		The complete equilibrium preserving reconstruction $\mathcal{W}$ is then
		obtained by the sum of the equilibrium and perturbation reconstruction
		\begin{equation}
		\label{eq:nm_2d_wb_0120}
		\Q^\text{rec}_{ij}(\x) = \mathcal{W} \left( \x ; \left\{  \hat\Q_{kl} \right\}_{(k,l) \in \mathcal S_{ij}} \right)
		= \Q^{eq}_{ij}(\x) + \delta \Q_{ij}(\x)
		.
		\end{equation}
		
		\begin{remark}
			(1) The method can be applied for arbitrary EoS by extending the modifications from \cref{sec:discrete_wb_eos} to two spatial dimensions in a straight forward way. 
			(2) The approximate well-balanced method presented in this section can be extended to three spatial dimensions without further complications.
		\end{remark}
		\section{Numerical experiments in two spatial dimensions}\label{sec:ne_2d}

		In all numerical experiments in this section we use the standard Roe flux~\cite{Roe1981}, Gauss-Legendre quadrature rules and the third order accurate CWENO3 reconstruction proposed in \cite{Levy2000b}. The semi-discrete schemes are evolved in time using a third-order accurate, four
		stage explicit Runge--Kutta method \cite{Kraaijevanger1991}.
		\refb{The two-dimensional LA method introduced in \cref{sec:wb_2d} is applied and compared to a two-dimensional standard method which is achieved by directly reconstructing the cell-averages as described in \cref{sec:fv_2d}.}
		\subsection{Two-dimensional polytrope}
		\label{sec:test_polytrope}
		In this test, we apply our two-dimensional well-balanced method (\la) on a two-dimensional polytrope.
		A polytrope is a hydrostatic configuration of an adiabatic gaseous sphere held together by self-gravitation.
		The test setup is given by \cite{Grosheintz2019}
		\begin{align}
		\label{eq:polytrope}
		\tilde \rho(\x)&:=\frac{\sin\left(\sqrt{2\pi}\, |\x|\right)}{\sqrt{2\pi}\, |\x|},
		&
		\tilde p(\x)&:=\tilde{\rho}(\x)^\gamma,
		&
		\vec g&:=-\nabla\phi(\x),
		&
		\phi(\x)&:=-2\frac{\sin\left(\sqrt{2\pi}\, |\x|\right)}{\sqrt{2\pi }\, |\x|},
		\end{align}
		where we choose $\gamma=2$ and the functions $\rho$ and $\phi$ are extended to $\x=0$ continuously. We set these initial conditions on Cartesian meshes for the domain $[-0.5,0.5]^2$ and use our third order accurate standard and \la method on these initial data. We evolve them until time $t=5$, which corresponds to approximately 6 sound crossing times. At the boundaries, we use Dirichlet boundary conditions. The resulting energy errors and convergence rates at different resolutions are presented in \cref{tab:2d_polytrope}. Using the well-balanced method significantly reduces the error, even though the gravity is not aligned with a coordinate direction in this setup. Moreover, as in previous tests, the increased order of accuracy on the hydrostatic solution is observed for the \la method.
		\begin{table}
			\centering
			\caption{
				\label{tab:2d_polytrope}
				$L^1$-errors in total energy for the 2-d polytrope described in \cref{sec:test_polytrope}.
			}
			
			\begin{tabular}{| c | c c | c c |}
				\hline
				\multirow{2}{*}{$N$} & \multicolumn{2}{c|}{Std-O3} & \multicolumn{2}{c|}{LA-O3}  \\
				& $E$ error & rate& $E$ error & rate\\
				\hline
				16  & 7.01e-05 & \multirow{2}{*}{3.0} & 9.58e-08 & \multirow{2}{*}{4.9} \\
				32  & 8.79e-06 & \multirow{2}{*}{3.0} & 3.29e-09 & \multirow{2}{*}{4.8} \\
				64  & 1.10e-06 &                      & 1.21e-10 &                      \\
				\hline
			\end{tabular}
			%
		\end{table}
		\subsection{Perturbation on the two-dimensional polytrope}
		\label{sec:test_polytrope_pert}
		
		As in \cite{Grosheintz2019}, we now add a perturbation to the polytrope to study if the application of our well-balanced methods can help resolving it more accurately.
		The initial pressure is perturbed in the following way
		$$
		p_\text{pert}(\x):=\left(1+A\,\exp\left(-\frac{|\x|^2}{0.05^2}\right)\right)\tilde p(\x).
		$$
		In our tests, we use different amplitude of the perturbation corresponding to $A=10^{-2},10^{-6},10^{-8}$. The spatial domain and numerical methods are the same as in \cref{sec:test_polytrope}. \refb{For comparison, simulations using the exactly well-balanced method from \cite{Berberich2020} have been added.} The final time is reduced to $t=0.2$, such that the perturbation can not reach the boundary. As a reference solutions to compute the errors we use simulations on a $512\times 512$ grid obtained with the well-balanced method introduced in \cite{Berberich2019}. Since this well-balanced method is exact on the hydrostatic background, the solutions are accurate enough to use them as reference. $L^1$ errors and convergence rates in total energy are presented in \cref{tab:2d_polytrope_pert} which show that the \la method is better at resolving the smaller perturbations than the standard method. \refb{The errors are comparable with the ones obtained from the exactly well-balanced method from \cite{Berberich2020} for the large and medium perturbation. On the small perturbation ($A=10^{-8}$) the exactly well-balanced method is more accurate than the \la method. The difference, however, is reduced on higher resolution due to the improved convergence of the \la method on the hydrostatic solution.} The convergence rates in density and momentum show a similar trend, hence we omit them for brevity. In \cref{fig:2d_polytrope_pert} the pressure perturbation for the test with $A=10^{-8}$ on the $256\times256$ grid is shown at final time, \refb{which again shows that the \la method is comparable to the exactly well-balanced method and thus significantly more accurate than the standard method.}
		\begin{table}
			\centering
			\caption{
				\label{tab:2d_polytrope_pert}
				$L^1$-errors and rates in total energy for perturbations of different size on the 2-d polytrope described in \cref{sec:test_polytrope_pert}. The third order accurate standard and \la methods are used.
			}
			\setlength\tabcolsep{3.3pt}
			\begin{tabular}{| c | *6{ c c |}}
				\hline
				\multirow{3}{*}{$N$} & \multicolumn{6}{c|}{$A=10^{-2}$} & \multicolumn{6}{c|}{$A=10^{-6}$} 
				\\
				\cline{2-13}
						& \multicolumn{2}{c|}{Std-O3} & \multicolumn{2}{c|}{LA-O3} & \multicolumn{2}{c|}{exact WB-O3}
						& \multicolumn{2}{c|}{Std-O3} & \multicolumn{2}{c|}{LA-O3} & \multicolumn{2}{c|}{exact WB-O3}
					 \\
				& $E$ error & rate& $E$ error & rate& $E$ error & rate& $E$ error & rate& $E$ error & rate& $E$ error & rate\\
				\hline
				64  & 2.92e-05 & \multirow{2}{*}{2.5} 
					& 3.83e-05 & \multirow{2}{*}{2.4}
					& 2.85e-05 & \multirow{2}{*}{2.5}
					& 2.63e-06 & \multirow{2}{*}{3.0} 
					& 4.28e-09 & \multirow{2}{*}{2.5} 
					& 2.85e-09 & \multirow{2}{*}{2.5}
					\\
				128 & 5.11e-06 & \multirow{2}{*}{3.0} 
					& 7.31e-06 & \multirow{2}{*}{3.0} 
					& 5.00e-06 & \multirow{2}{*}{3.1}
					& 3.35e-07 & \multirow{2}{*}{3.0} 
					& 7.58e-10 & \multirow{2}{*}{3.0} 
					& 5.01e-10 & \multirow{2}{*}{3.1}
					\\
				256 & 6.18e-07 &                      
					& 9.38e-07 &                      
					& 6.04e-07 &      
					& 4.19e-08 &               
					& 9.53e-11 &                      
					& 6.04e-11 &                     
					\\
				\hline
			\end{tabular}
			\\[1em]
			\begin{tabular}{| c | *3{ c c |}}
				\hline
				\multirow{3}{*}{$N$} & \multicolumn{6}{c|}{$A=10^{-8}$}
				\\
				\cline{2-7}
				& \multicolumn{2}{c|}{Std-O3} & \multicolumn{2}{c|}{LA-O3} & \multicolumn{2}{c|}{exact WB-O3}
				\\
				& $E$ error & rate& $E$ error & rate& $E$ error & rate\\
				\hline
				64  
				& 2.60e-06 & \multirow{2}{*}{3.0} 
				& 1.27e-09 & \multirow{2}{*}{4.0}
				& 2.85e-11 & \multirow{2}{*}{2.5}
				\\
				128 
				& 3.31e-07 & \multirow{2}{*}{3.0} 
				& 7.76e-11 & \multirow{2}{*}{4.0}
				& 5.01e-12 & \multirow{2}{*}{3.1}
				\\
				256                 
				& 4.17e-08 &                      
				& 4.79e-12 &                      
				& 6.04e-13 &                      
				\\
				\hline
			\end{tabular}
			
		\end{table}
		\begin{figure}
			\centering
			\includegraphics[scale=1]{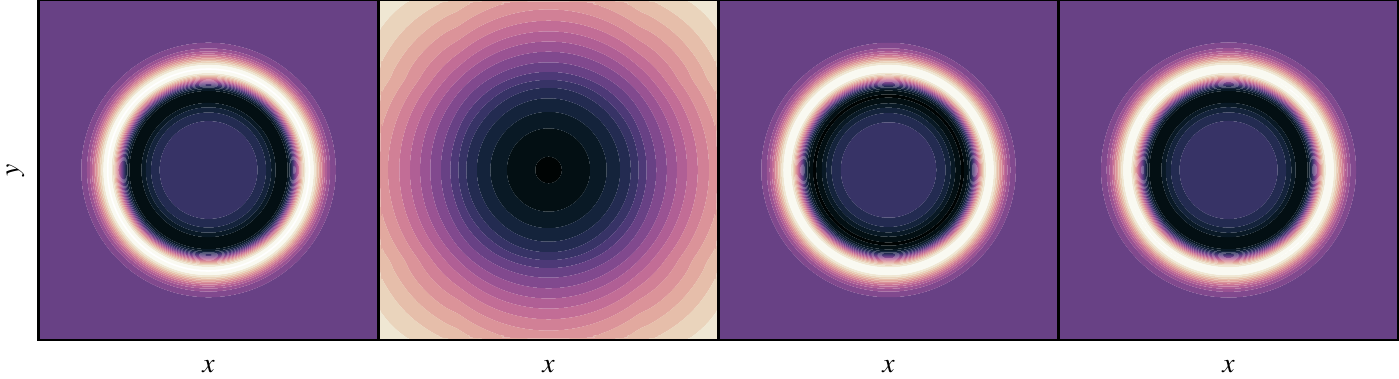}\\
			\vspace{0.7em}
			\includegraphics[scale=1]{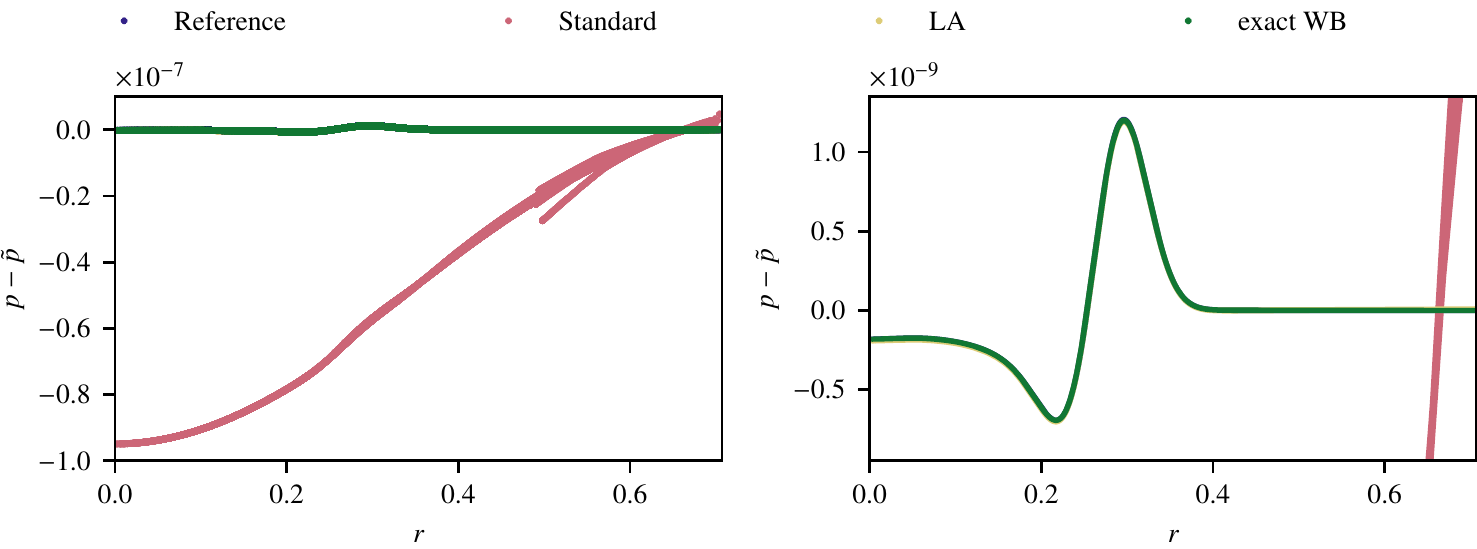}
			\caption{%
				\label{fig:2d_polytrope_pert}%
				Pressure perturbation $p-\tilde p$ from the 2-d polytrope for the initial perturbation with $A=10^{-8}$ on a $256\times256$ cells grid after time $t=0.2$. The test is described in \cref{sec:test_polytrope_pert}. The reference solution is obtained using the exactly well-balanced method from \cite{Berberich2019} on a $512\times512$ cells grid. In the top panels, the color ranges (dark to light) from -0.8e-9 to 1.2e-9 in the first and third, and fourth plot (reference, \la method, and the exactly well-balanced method from \cite{Berberich2020}) and from -1e-7 to 1e-8 in the central plot (standard method). The full domain $[-0.5,0.5]^2$ is shown. The bottom panels show the pressure perturbation over radius in a scatter plot. The difference between the two bottom panels is the range of the values at the $y$-axis. In the bottom plots, no difference can be seen between the the data from all simulations except the one using the non-well-balanced standard method.}
		\end{figure}
		
		\subsection{Efficiency of the \tla}
		\label{sec:efficiency}
		To compare the computational effort of the \la method to a standard method, we use the test case presented in \cref{sec:test_polytrope_pert} with different sizes of perturbations given by $\eta=10^{-4}$ and $\eta=10^{-8}$. The test case is run with the standard method and the \tla method with different grid resolutions ($N = 16,24,32,48,64,96,128,192$). Each of the tests is repeated $10$ times to account for fluctuations in the machine's performance. The results of these tests are visualized in \cref{fig:efficiency}. For both sizes of perturbation using the \la method is significantly more efficient. For the larger perturbation, e.g., the computation time necessary to obtain a certain accuracy is reduced by one to two orders of magnitude by using the \la method. As expected, the difference is larger for the smaller perturbation: The computation time to obtain a certain accuracy is for this perturbation reduced by two to four orders of magnitude.
		\begin{figure}
			\centering
			\includegraphics[width=\textwidth]{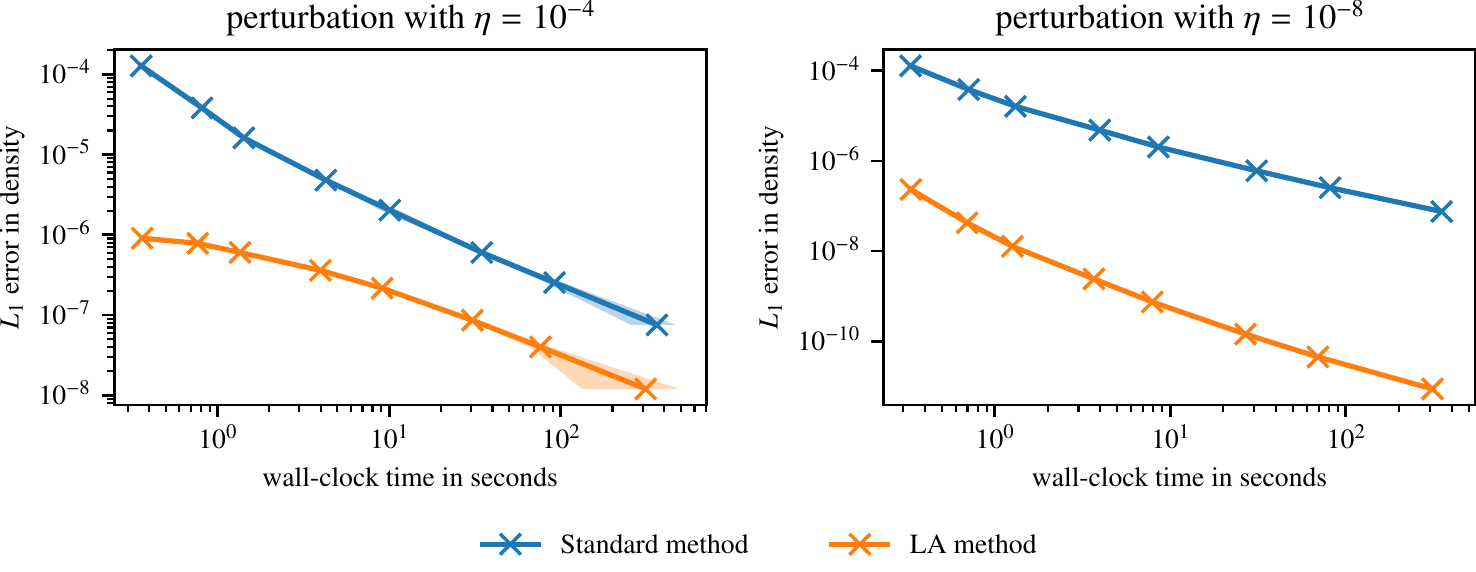}
			\caption{\label{fig:efficiency}%
				$L_1$ density error over mean wall-clock time needed to compute the result for the test introduced in \cref{sec:test_polytrope_pert} with different sizes for the perturbation (left pane: $\eta=10^{-4}$, right panel: $10^{-8}$).  Each of the crosses is obtained by computing the mean value of 10 repetitions of the same test at a given resolution  as described in \cref{sec:efficiency}.
				The shaded area shows the area of the variance in computation time. For both tests using the \tla method requires less computational effort to obtain the same accuracy than in the standard method.
			}
		\end{figure}
		
		\subsection{Radial Rayleigh--Taylor instability}
		\label{sec:rrt}
		
		In this test, we use a piece-wise isothermal hydrostatic state in the two-dimensional gravitational potential 
		\begin{equation*}
		\phi(\x):=-20\frac{\sin\left(\sqrt{2\pi}\, |\x|\right)}{\sqrt{2\pi}\, |\x|}
		\end{equation*}
		and gravitational acceleration 
		$\vec g:=-\nabla\phi(\x)$.
		The initial data are given by
		\begin{align*}
		(\rho,p)(\x)&:=
		\begin{cases}
		(\tilde\rho_\text{in},\tilde p_\text{in})(\x) & \text{ if }\|\x\|_2<r_0,\\
		(\tilde\rho_\text{out},\tilde p_\text{out})(\x)  & \text{else},
		\end{cases}
		&
		\vel(\x)&:=\begin{pmatrix}
		0\\0
		\end{pmatrix},
		\end{align*}
		where
		\begin{align*}
		\tilde{\rho}_\text{in}(\x) &:= a c \exp\left(-a\phi(\x)\right),&
		\tilde{\rho}_\text{out}(\x) &:= b \exp\left(-b\phi(\x)\right),\\
		\tilde{p}_\text{in}(\x) &:= c \exp\left(-a\phi(\x)\right),&
		\tilde{p}_\text{out}(\x) &:= \exp\left(-b\phi(\x)\right),
		\end{align*}
		and $c = \exp\left( (a-b)\phi\left((r_0,0)^T\right) \right)$. Choosing $b>a$ makes the system unstable, such that Raleigh--Taylor instabilities are expected to develop \cite{Chandrasekhar1961}.
		
		For the numerical computations, we use the above initial data with $r_0=0.2$ and $(a,b)=(1,2)$ in the domain  $[0,0.5]^2$, and evolve them until time $t=0.6$. We use the third order accurate standard and \la method on a $64\times 64$ cells grid. At the $x=0$ and $y=0$ boundaries we use wall-boundary conditions which are consistent with the symmetry of the problem. At the outer boundaries we extrapolate 
		$\left(
		\rho-\tilde\rho_\text{out},\,
		\rho u,\,
		\rho v,\,
		E - (\gamma-1)\tilde p_\text{out} 
		\right)^T$ 
		in order to not destroy the hydrostatic solution at the boundary.
		The results are visualized in \cref{fig:rrt_contourf,fig:rrt_scatter}. The simulation with the standard method crashes approximately at time $t\approx0.2981$. In \cref{fig:rrt_scatter} the relative density deviations $\rho/\tilde\rho_\text{initial}-1$ from the initial density over radius are shown at time $t=0.298$. While the \la method is capable of accurately maintaining the hydrostatic solution away from the discontinuity, there are significant spurious perturbations for the standard method, especially at the outer boundary. In \cref{fig:rrt_contourf}, the relative density deviation $\rho/\tilde\rho_\text{out}$ is presented at different times. As expected, Rayleigh--Taylor instabilities appear at the interface between the light and the dense fluid.
		\begin{figure}
			\centering
			\includegraphics[width=0.49\textwidth]{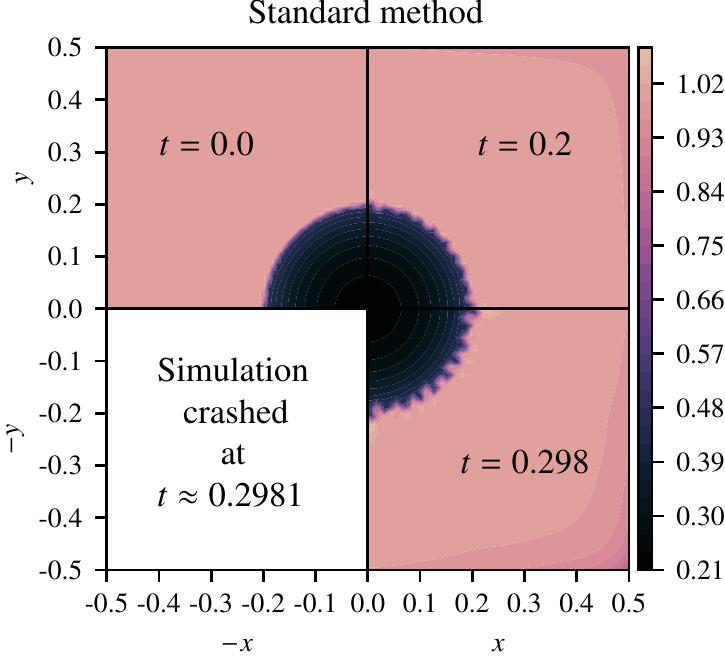}
			\hfill
			\includegraphics[width=0.49\textwidth]{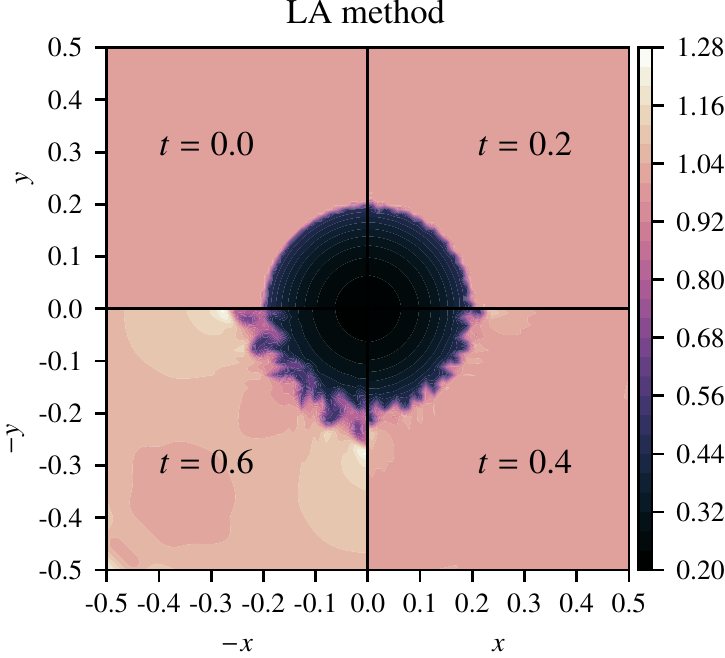}
			\caption{%
				\label{fig:rrt_contourf}%
				Relative density deviation $\rho/\tilde\rho_\text{out}$ from the outer hydrostatic state of the radial Rayleigh--Taylor instability. Setup and method are described in \cref{sec:rrt}. Different times using the standard (left) and the \la method (right). The simulation using the standard method crashes at $t\approx0.2981$.
			}
		\end{figure}
		\begin{figure}
			\centering
			
			\includegraphics[scale=1]{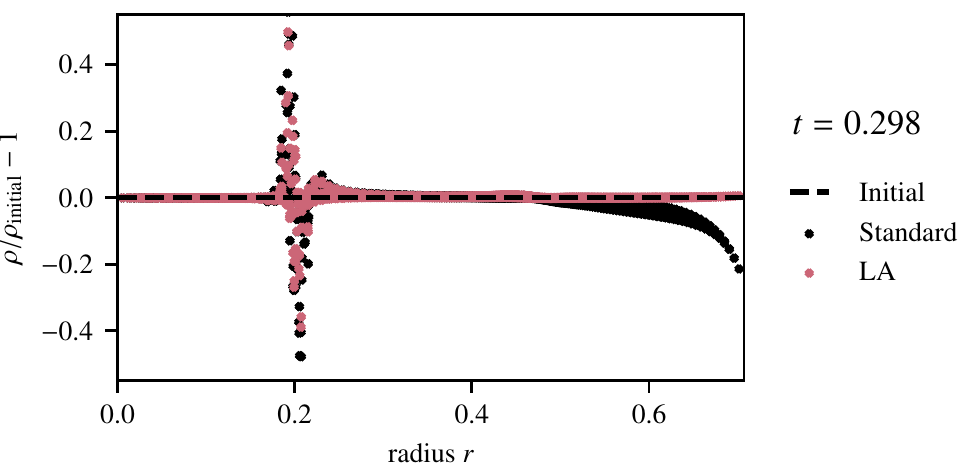}
			\caption{%
				\label{fig:rrt_scatter}%
				Relative density deviations $\rho/\tilde\rho_\text{initial}-1$ from the initial density stratification over radius at time $t=0.298$ for the standard and \la method. The simulation using the standard method crashes at $t\approx0.2981$.
			}
		\end{figure}
		\section{Conclusions and outlook}\label{sec:conclusions}
		
		Novel well-balanced high order accurate finite volume methods have been introduced in this work. The first method (\dwb method) exactly balances a high order discretization of any hydrostatic state with any EoS. The stencil of the method, however, is larger than the stencil of a standard method with the same order of accuracy. Localizing the high order approximation to the hydrostatic state in each cell (\la method) reduced the stencil to that of a comparable standard method. For the \la method the well-balanced property can not be shown analytically. However, in numerical tests it is found to be even more accurate on analytical hydrostatic solutions than the \dwb method. The \la method shows superconvergence when computing hydrostatic solutions with the ideal gas EoS; the convergence rate is increased by two. This was observed not only in one spatial dimension but also in the two-dimensional extension. 
		The high order accuracy and the robustness of the \dwb and \la method have been shown numerically. Also, numerical tests verified the capability of the methods to accurately capture small perturbations on hydrostatic states.
		
		We conclude that the proposed well-balanced methods, in particular the \la method, are especially useful in situations, in which no knowledge about the structure of the hydrostatic states, which will appear in the simulation, is available. It can be implemented in an existing code and can be used as the default method which significantly improves the accuracy at any hydrostatic state.

\section*{Acknowledgments}
Roger K\"appeli acknowledges support by the Swiss National Science Foundation (SNSF) under grant 200021-169631.
Praveen Chandrashekar would like to  acknowledge support  from the Department of Atomic Energy,  Government of India, under project no.   12-R\&D-TFR-5.01-0520.

\bibliographystyle{model1-num-names}
\bibliography{library}

\appendix
\section{Equilibrium reconstruction for general EoS}
\label{sec:appendix_wb_rec}

In the following we discuss the existence and uniqueness of a cell center
equilibrium pressure $p_{0,i} > 0$ solving \cref{eq:appendix_wb_rec_0010}.
\paragraph{Existence} We only have to show that the discrete hydrostatic pressure approximation $p^\text{eq}$ is positive. Since the actual pressure $p$ is assumed to be positive in the Euler equations, the domain $\Omega$ is compact, and the pressure is continuous in the hydrostatic state, there is a minimal pressure value $p_\text{min}$. The discrete hydrostatic pressure approximation has an error $\|p^\eq-p\|_{l^1}=\mathcal O\left(\Delta x^m\right)$. Consequently, for sufficiently small values of $\Delta x$ we have $\|p^\eq-p\|_{l^1}<p_\text{min}$ which implies $p^\text{eq}>0$.
\paragraph{Uniqueness}
Notice that the derivative of internal energy density with respect to
pressure at constant density in \cref{eq:appendix_wb_rec_0040} is a
fundamental EoS property.
This expression is related to the so-called Gr\"{u}neisen coefficient
$\Gamma$
\begin{equation}
\label{eq:appendix_wb_rec_0060}
\Gamma = \left( \frac{\del p}{\del \varepsilon} \right)_{\rho}
,
\end{equation}
which measures the spacing of the isentropes in the $p$-$V$-plane
($V = 1/\rho$ is the so-called specific volume).
The Gr\"{u}neisen coefficient is a characteristic EoS variable and it is
positive away from phase transitions (see e.g.\ \cite{MenikoffPlohr1989}).
So, if we assume that the quadrature weights are positive, the function's
derivative \cref{eq:appendix_wb_rec_0040} will always keep the same sign away
from a phase transition.
Therefore, the function whose root we are seeking
\cref{eq:appendix_wb_rec_0020} is a strictly monotone function in the pressure
variable and, if a root exists, it is unique.

\end{document}